\documentclass[11pt]{amsart}
\usepackage{latexsym, hyperref}
\usepackage{amsfonts,amssymb,amsmath,amsthm,amscd}
\usepackage{mathrsfs}  

\makeatletter
\def\numberwithin#1#2{\@ifundefined{c@#1}{\@nocnterrr}{%
  \@ifundefined{c@#2}{\@nocnterr}{%
  \@addtoreset{#1}{#2}%
  \toks@\expandafter\expandafter\expandafter{\csname the#1\endcsname}%
  \expandafter\xdef\csname the#1\endcsname
    {\expandafter\noexpand\csname the#2\endcsname
     .\the\toks@}}}}
\makeatother
\numberwithin{equation}{section}

\newcommand{\diag}{\operatorname{diag}}
\newcommand{\Ad}{\operatorname{Ad}}
\newcommand{\tr}{\operatorname{tr}}

\newcommand{\Ric}{\operatorname{Ric}}

\newcommand{\hess}{\operatorname{Hess}\,}

\newcommand{\rank}{\operatorname{rank}}
\newcommand{\dd}{\mathbf{d}}
\newcommand{\qq}{\mathbf{q}}
\newcommand{\pp}{\mathbf{p}}
\newcommand{\ww}{\mathbf{w}}
\newcommand{\vv}{\mathbf{v}}
\newcommand{\aaa}{\mathbf{a}}
\newcommand{\cc}{\mathbf{c}}

\theoremstyle{plain} \newtheorem{theo}{Theorem}[section]
\theoremstyle{remark} \newtheorem{rem}[theo]{\textbf{Remark}}
\theoremstyle{plain} 
\theoremstyle{plain} \newtheorem{lem}[theo]{Lemma}
\theoremstyle{plain} 
\theoremstyle{plain} \newtheorem{pro}[theo]{Proposition}
\theoremstyle{definition} 
\theoremstyle{definition} \newtheorem*{notation}{Notation}
\theoremstyle{definition} \newtheorem*{ackn}{Acknowledgements}

\begin{document}

\title{Painlev\'e Analysis of Ricci solitons over warped products} 

\author{Alejandro Betancourt de la Parra}
\address{Centro de Investigaci\'on en Matem\'aticas, A.C., Jalisco S/N, Col. Valenciana CP: 36023 Guanajuato, Gto, M\'exico}
\email{alejandro.betancourt@cimat.mx}
\date{revised \today}

\begin{abstract}
We carry out a Painlev\'e analysis to find the cases where the cohomogeneity one steady Ricci soliton equation can be integrable. We concentrate on two classes of solitons: warped products and complex line bundles over a Fano K\"ahler-Einstein base. For warped products, the analysis singles out the case with one factor where the dimension of the hypersurface is a perfect square, with the $n=4$ particularly distinguished. The case with two factors each of dimension $2$ is also singled out by the analysis. In the case of complex line bundles, a 1-parameter family is singled out for every even dimension.
\end{abstract}

\maketitle 

\section{Introduction}
A gradient Ricci soliton (or GRS for short) is a triple $(M,g,u)$, where $M$ is a smooth manifold (which we assume without boundary but possibly non-compact), $g$ a Riemannian metric, and $u$ a smooth function such that the identity
\begin{equation}\label{soleq}
 \Ric + \hess (u) +  \lambda g=0,
\end{equation} 
is satisfied, where $\Ric$ is the Ricci tensor of $g$, $\hess (u)$ is the Hessian with respect to the Levi-Civita connection, and $\lambda$ is a real number. The soliton is called expanding, steady, or shrinking whenever $\lambda$ is positive, zero, or negative, respectively. The function $u$ appearing in the definition is called soliton potential and is determined up to a constant. Ricci solitons play a fundamental role in understanding of singularity formation of the Ricci flow.

The soliton equation \eqref{soleq} constitutes a natural generalization of the Einstein equation. As such, methods originally devised to study and produce examples of Einstein metrics are often useful in the context of Ricci solitons. A natural approach to Einstein metrics is to assume a symmetry condition to reduce the equation to a more manageable form. The most symmetric case, that of homogeneous spaces, has been studied extensively and numerous existence and non-existence results are well known today (see for example \cite{besse}). In contrast to the Einstein case, all homogeneous gradient Ricci solitons are rigid, meaning that they are Einstein if compact, or a quotient of an Einstein metric times the Gaussian soliton if non-compact. The most symmetric nontrivial gradient Ricci solitons are of cohomogeneity one, that is, they have a group acting by isometries with a one dimensional orbit space. This is precisely the setting in which the soliton equation was studied in \cite{DW2011, BDW2016}. 

Another important difference Einstein metrics and GRS's is that there are various examples of explicit closed form formulas for the metric in the former case whereas there are very few examples of the latter. In particular, most of these explicit solitons are of K\"ahler type. To our knowledge, the only explicit non-K\"ahler solitons are those in \cite{BDW2016}, where explicit formulas for the Bryant soliton and for solitons over double warped products are given (in both cases the formulas only exist in dimension 5). These formulas were obtained by reformulating the soliton equation as a Hamiltonian system with constraint. This Hamiltonian is obtained from Perelman's $\mathcal{F}$ and $\mathcal{W}$ functionals \cite{perelman1} in a manner analogous to the Hamiltonian formulation of General Realtivity in the ADM formalism \cite{DW2001}. 

The objective of the present paper is to analyze the Hamiltonian system corresponding to the cohomogeneity one steady GRS equation for multiple warped products and for the B\'erard Bergery ansatz (see section \ref{setup}) in order to determine whether there are more cases where the equations can be integrated explicitly. Steady solitons of these forms were shown to exist by Bryant \cite{bryant}, Ivey \cite{Ivey1994}, Cao \cite{Cao1996}, and Dancer and Wang \cite{DW2009,DW2011}. In \cite{BDW2016} we managed to obtain explicit formulas for the Bryant soliton and for doubly warped products when the dimension of the underlying manifold is 5. In the case of the Bryant soliton, the key to obtaining the formula was finding a non-trivial conserved quantity for the Hamiltonian. For the doubly warped product, the main element was the construction of a superpotential (i.e., a solution of the Hamilton-Jacobi equation). In both cases, the constructions depended crucially on the dimension being exactly 5. In this paper we employ the technique of Painlev\'e analysis to identify the cases where the Hamiltonian system is integrable. One expects to find explicit closed form formulas for steady GRS's precisely in these cases. 

Before continuing, we recall that the Painlev\'e test was first pioneered by S. Kowalewski in the nineteenth century on her work on spinning tops and more recently formalized in \cite{ARS1980,ARS1980b}. The idea behind this procedure is to identify the cases in which a system of ODE's has solutions that have the Painlev\'e property, that is, where the general solution is meromorphic with movable singularities (we assume the system to be complex). The existence of such solutions is strongly associated with integrability of the system. The steps in this procedure are summarized as follows: 

\begin{enumerate}
	\item Assume that there exists a meromorphic solution and determine the\textit{ leading order} terms.
	\item Find the \textit{resonances} of the system, that is, the steps of the expansion at which free parameters may enter.
	\item Check the \textit{compatibility} conditions at each resonance to verify that the recursion relation can be solved at every step.
	\item Verify that the series \textit{converges} in a punctured neighborhood of the singularity.
\end{enumerate}
It is usually said that a system passes the Painlev\'e test if the general solution is meromorphic with movable singularities. In this case, the series are expansions in integer powers of some variable $t$ and have a maximal number of free parameters. We will refer to this property as the \textit{strong Painlev\'e property}. We will also consider series expansions which are meromorphic on the variable $z=t^{1/q}$ for an integer $q$. If such series have a maximal number of free parameters we will say that the system has the \textit{weak Painlev\'e property}.

We describe the contents of this paper. In section \ref{setup} we recall the Hamiltonian formulation from \cite{BDW2016} and introduce new variables which are more amenable to Painlv\'e analysis. We obtain systems of ODE's that correspond to multiple warped products and the B\'erard Bergery ansatz respectively.

Sections \ref{onefactor} and \ref{arbfactor} are dedicated to the analysis of the systems corresponding to warped product metrics with one and an arbitrary number of factors, respectively. In the first case, we distinguish two families of series solutions only one of which passes the weak Painlev\'e test (whenever the dimension of the underlying manifold is of the form $n^2+1$). Furthermore, this family passes the strong test when $n=2$. In the next section we show there are various families of series solutions which can be parametrized by the set of rational points on a certain ellipsoid. None of these series has the weak Painlev\'e property, however the case where there are two factors, each of dimension 2, is singled out by the analysis.

In section \ref{berard} we analyze the B\'erard Bergery case. We show that the corresponding system fails the weak Painlev\'e test. As before, the analysis singles out a one parameter family of solutions which suggests that a subsystem of the original may be integrable.

Finally, section \ref{conclusions} discusses the convergence of the formal expansions of previous sections as well as the geometric properties of the corresponding solitons. We mention that the analogous study of the integrability properties of the Einstein equations using Painlev\'e analysis has been carried out in \cite{DW2001}, \cite{DW2001b}, \cite{DW2003}, and \cite{DW2003b}.

\begin{ackn}
I would like to thank Prof. Andrew Dancer for suggesting the problem and for thoroughly revising this work. My gratitude also goes to Prof. McKenzie Wang for his helpful comments.
\end{ackn}

\section{Setting up the equations} \label{setup}
We consider a Riemannian manifold $(M,g)$ of dimension $N=n+1>2$ with an isometric cohomogeneity one action by a compact Lie group $G$. In this setting the generic orbit is of the form $G/K$ for a closed subgroup $K$ and has dimension $n$. We further assume there exists a singular orbit of the form $G/H$ for a closed subgroup $K\subset H $. We denote by $M_0$ the subset of $M$ consisting of generic orbits, so that $M_0$ is diffeomorphic to $I \times G/K$ for an open interval $I$. Recall that a cohomogeneity one manifold can have at most two singular orbits; in our setting one is placed without loss of generality at $t=0$, and $M$ is compact if and only if there is a second singular orbit at another $t>0$. On $M_0$ the metric $g$ can be written as
$$g=dt^2+g_t,$$
where the parameter $t$ is the distance to the singular orbit at $t=0$, or equivalently, the length of a geodesic that intersects the generic orbits orthogonally. For the time being, we will only study solutions to the soliton equation on $M_0$. There are additional conditions that need to be satisfied at the singular orbits in order for the metric to extend to all of $M$ (cf. section 2 in \cite{EW2000}). Likewise, if there is only one singular orbit there are conditions that guarantee that $g$ is complete.\\

We can think of $g_t$ as a family of $G$-invariant metrics on $G/K$. By fixing an $\Ad(K)$-invariant decomposition of the Lie algebra
$$ \mathfrak{g}=\mathfrak{k}\oplus \mathfrak{p}$$
(so that $\mathfrak{p}$ is isomorphic to $T_{[K]}G/K$ through the exponential map) we can identify such metrics with $\Ad(K)$-invariant metrics on $\mathfrak{p}$. If we fix a background $\Ad(K)$-invariant metric $Q$ on $\mathfrak{p}$, $g_t$ can be identified with a family $q_t$ of symmetric, positive definite, $\Ad(K)$-invariant endomorphisms of $\mathfrak{p}$ by taking $g_t$ on $T_{[K]}G/K$ to correspond to $Q\circ q_t$ on $\mathfrak{p}$. We denote such endomorphisms by $S_+^2(\mathfrak{p})^K$. A Ricci soliton can then be identified with a path $(q_t,u(t))$ in the configuration space $\mathscr{C}:= S_+^2(\mathfrak{p})^K\times \mathbb{R}$. Here we are assuming that the soliton potential $u$ only depends on $t$ (this is always satisfied in the cohomogeneity one setting).\\

The Hamiltonian in \cite{BDW2016} was defined on the cotangent bundle of the configuration space $T^*\mathscr{C}$. The intuition behind its definition is that the integral curves of the canonical equations (lying on the zero energy set) correspond, under the inverse Legendre transformation, to the Euler-Lagrange equations of a certain Lagrangian on $T\mathscr{C}$. In the case of steady solitons this Lagrangian is derived from Perelman's $\mathcal{F}$ energy functional, which plays a role akin to that of the total scalar curvature for Einstein metrics. We give an account of this construction in what follows.\\

First, we will assume that  the isotropy representation of $G/K$ splits into pairwise inequivalent irreducible subrepresentations. This is not necessary to obtain the Hamiltonian, but it simplifies computations considerably. In this case we can decompose $\mathfrak{p}$ into irreducible, $\Ad(K)$-invariant, $Q$-orthogonal summands 
$$ \mathfrak{p}=\mathfrak{p}_1 \oplus \cdots \oplus \mathfrak{p}_r,$$
each of which has dimension $d_i=\dim \, \mathfrak{p}_i$. This decomposition is unique up to isotypical summands and allows us to introduce coordinates on $\mathscr{C}$ by writing each $q \in S_+^2(\mathfrak{p})^K$ as $q=\diag (e^{q_1}I_{d_1}, \ldots ,e^{q_r}I_{d_r}$), where $I_{d_i}$ is the identity operator on $\mathfrak{p_i}$. This induces global coordinates $(q_1, \ldots, q_r, u, p_1, \ldots , p_r, \phi)$ on $T^* \mathscr{C}$, where $\pp=( p_1, \ldots , p_r, \phi)$ are the conjugate momenta associated to $\qq=(q_1, \ldots, q_r, u)$. These coordinates can be used to write the Hamiltonian down explicitly. Before doing so, we recall that the scalar curvature of $G/K$ with respect to $g_t$ can be written (cf. equation (1.3) in \cite{WZ1986}) in terms of $\qq=\qq(t)$ as
$$ S(\qq)  =\sum_{\ww \in \mathcal{W}} A_\ww e^{\ww \cdot \qq},$$
where $\mathcal{W}$ is a set of weight vectors in $\mathbb{Z}^{r+1}$ depending only on $G/K$ and $A_\ww$ are constants. Weight vectors have at most three nonzero entries and can be of three types:
\begin{enumerate}
	\item Type I vectors, where the $i$-th entry is -1 and the last is zero. We denote this vector as $(-1^{(i)},0)$. For type I vectors $A_\ww>0$.
	\item Type II vectors, where the $i$-th entry is 1, the $j$-th and $k$-th are -1, and the last one is zero. We denote this vector as $(1^{(i)},-1^{(j)},-1^{(k)},0)$. Type II vectors have $A_\ww<0$.
	\item Type III vectors, where the $i$-th entry is -1 and the $j$-th is -2. We denote this vector by  $(1^{(i)},-2^{(j)},0)$. Type III vectors have $A_\ww<0$.
\end{enumerate}
In this setting, the Hamiltonian \cite{BDW2016} corresponding to a cohomogeneity one steady soliton is given by
\begin{equation}\label{hamiltonian}
H(\qq,\pp)= e^{-\frac{1}{2} \dd \cdot \qq} J(\pp)-e^{\frac{1}{2} \dd \cdot \qq} \left( E+\sum_{\ww \in \mathcal{W}} A_\ww e^{\ww \cdot \qq}\right),
\end{equation}
where $E$ is a constant (it's exact value is unimportant for now, but it coincides with the Perelman energy $E=\mathcal{F}(g,u)$ on steady solitons), $\dd=(d_1, \ldots , d_r, -2)$, the generic orbit has dimension $n=d_1+\ldots+d_r$, and 
$$ J(\pp)= -\left( \sum_{i=1}^r \frac{p_i^2}{d_i}+ \sum_{i=1}^r p_i\phi + \frac{n-1}{4}\phi^2\right).$$
Solutions of the canonical equations for $H$ lie on level sets of the Hamiltonian. As we will see, those lying on $H^{-1}(0)$ correspond to solutions of the steady Ricci soliton equation on $M_0$.\\

The cohomogeneity one Ricci soliton equation is equivalent to the system of ODE's
\begin{align}
r_t-\dot{L_t}-(\tr L_t- \dot{u})L_t+\lambda I=& 0 \label{primera}\\
-\tr(L_t^2)-\tr (\dot{L_t})+\ddot{u}+\lambda=& 0 \label{segunda}\\
d(\tr L_t)+ \delta^{\nabla^t} L_t =& 0, \label{tercera}
\end{align}
where $r_t$ is the Ricci endomorphism of $g_t$, $L_t$ is the shape operator of the hypersurface $\{ t \} \times (G/K)$, and $\delta^\nabla: T^*(G/K) \otimes T\,(G/K)\rightarrow T\,(G/K)$ is the codifferential. This system is obtained by evaluating \eqref{soleq} respectively on tangent, normal, and mixed directions with respect to the orbits $G/K$ (cf. \cite{DW2011}). Theorem 5 in \cite{BDW2016} states that equations \eqref{primera}-\eqref{segunda} correspond to the canonical equations of $H$ lying on the zero energy set $\{H=0 \}$.
\begin{pro}
Consider the Hamiltonian $H$ in \eqref{hamiltonian} defined on the cotangent space $T^*\mathscr{C}$ of the configuration space. The integral curves of the canonical equations of $H$ lying on the $\{H=0 \}$ set correspond, under the inverse Legendre transform, to solutions of the tangent and normal parts of the Ricci soliton equation \eqref{primera}-\eqref{segunda}.
\end{pro}
On the other hand, the equation for mixed directions \eqref{tercera} is automatically satisfied whenever \eqref{primera}-\eqref{segunda} hold and we can further establish the $C^3$ regularity of the solution (cf. Proposition 3.19 in \cite{DW2011}). One of the advantages of looking at \eqref{primera}-\eqref{tercera} is that these equations are valid not just for the cohomogeneity one setting but for manifolds containing a dense open set foliated by an equidistant family of hypersurfaces where $L$, $r$, and $u$ depend only on $t$. This is satisfied when $M$ is a multiple warped product of Einstein manifolds of positive scalar curvature or on the so called B\'erard Bergery ansatz to be discussed in the following sections.

\subsection{Warped products}
We start by considering the Hamiltonian when the hypersurface is a product of Einstein manifolds $(M_i,g_i)$ of positive scalar curvature. Examples of such solitons have been found by Bryant \cite{bryant}, Ivey \cite{Ivey1994}, and more recently Dancer and Wang \cite{DW2009}. In what follows, we let $d_i \geq 2$ be the dimension of $M_i$. We write the metric on $I\times M_1 \times \ldots \times M_r$ as
\begin{equation}\label{metric}
 g=dt^2+g_t=dt^2+ e^{q_1} g_1 + \ldots + e^{q_r} g_r.
\end{equation}
In this case $\mathcal{W}$ is composed of the type I vectors $\ww_i=(-1^{(i)},0)$ for $1\leq i \leq r$. We will introduce new variables so the canonical equations of $H$ take a form that is more suitable for Painlev\'e analysis. In particular, we will get rid of the exponential terms and obtain a polynomial system with quadratic nonlinearities analogous to \cite{DW2001}. In what follows we can assume without loss of generality that $A_i=E=1$ for all $i$.

Given the form of the Hamiltonian, it is natural to introduce 
\begin{equation*}
\begin{aligned}
x_i&:=e^{\ww_i \cdot \qq} \qquad \textnormal{for } 1\leq i \leq r \\
x_{r+1}&:= e^{\frac{1}{2}\dd \cdot \qq}
\end{aligned}
\end{equation*}
to get rid of the exponential terms. Define the matrix whose rows are
$$ U:=\left( \begin{array}{c} 
\ww_1 \\ \vdots \\ \ww_r \\ \dd /2
\end{array} \right), $$
so change of coordinates above is given by $x_i= e^{U_{i k}q_{k}} $  for $1 \leq i \leq r+1$. Here we are making the identification $q_{r+1}=u$ and using the Einstein convention for summation to simplify the notation. The corresponding momentum coordinates are given by
\begin{equation} \label{conjugatemomenta}
 y_i:= \frac{\sum_{k=1}^{r+1} p_k U^{ki}}{x_i}
\end{equation}
where $U^{ki}$ denotes the $k,i^{th}$-entry of the inverse matrix $U^{-1}$. As before, here we are making the identification $p_{k+1}=\phi$. If we set $M:= UJU^T$, the kinetic energy term $J_{ij}p_i p_j$ in the Hamiltonian \eqref{hamiltonian} becomes $M_{kl}x_k y_k x_ly_l$. Furthermore, notice that $M$ is a diagonal matrix $M=\diag(-d_1^{-1},\ldots, -d_r^{-1},1/4)$, so the Hamiltonian \eqref{hamiltonian} takes the form
\begin{equation*}
H= -\frac{1}{x_{r+1}} \left( \sum_{k=1}^r  \frac{x_k^2 y_k^2}{d_k} \right) +\frac{x_{r+1}y_{r+1}^2}{4} - x_{r+1} \left(\sum_{k=1}^r x_k +1 \right).
\end{equation*}
In these coordinates the canonical equations become
\begin{equation} \label{canonical}
\begin{array}{lll}
\dot{x}_i  = \displaystyle -\frac{2x_i^2 y_i}{d_i x_{r+1}} & \quad \dot{y}_i =\displaystyle \frac{2x_i y_i^2}{d_i x_{r+1}}+x_{r+1} &  \textnormal{ for } \quad 1\leq i \leq r \\[4mm]
\dot{x}_{r+1} = \displaystyle\frac{x_{r+1} \, y_{r+1}}{2} & \quad \dot{y}_{r+1} =-\displaystyle \sum_{k=1}^r \frac{x_k^2 y_k^2}{d_k x_{r+1}^2}-\frac{y_{r+1}^2}{4}+\sum_{k=1}^r x_k +1. \\
\end{array}
\end{equation}
The equations in this form are still too complicated. To simplify them further we note two things. First, we are only interested in solutions lying on the level set $\{ H=0 \}$. The last equation in \eqref{canonical} is greatly simplified by this fact. Second, we can introduce the auxiliary variables
\begin{align*}
u_i&:= \frac{x_i\, y_i}{d_i x_{r+1}} \quad \textnormal{for } 1\leq i \leq r \\[2mm]
u_{r+1}&:=\frac{y_{r+1}}{2}.
\end{align*}
to rewrite both $H$ and \eqref{canonical} in a more manageable form. In these new variables the Hamiltonian takes the much simpler form
\begin{equation*}
H=-x_{r+1}\left( \sum_{k=1}^r d_k u_k^2 -u_{r+1}^2+ \sum_{k=1}^r x_k +1 \right),
\end{equation*}
and the system \eqref{canonical} restricted to $H=0$, becomes
\begin{equation} \label{xu}
\begin{array}{llc}
\dot{x}_i  =-2x_i u_i & \quad \dot{u}_i =-u_i u_{r+1}+\displaystyle\frac{1}{d_i} x_i &  \textnormal{ for } \quad 1\leq i \leq r \\[3mm]
\dot{x}_{r+1} = x_{r+1} \, u_{r+1} & \quad \dot{u}_{r+1} =-\displaystyle\sum_{k=1}^{r}d_k u_k^2. \\
\end{array}
\end{equation}
Note that $x_{r+1}$ is an exponential so it never vanishes on $I$ (although the limit as $t\rightarrow 0$ can be zero). Thus $H=0$ if and only if $ \sum_{k=1}^r d_k u_k^2 -u_{r+1}^2+ \sum_{k=1}^r x_k +1 $ vanishes. It will be useful in the future to notice this expression is constant on solutions of \eqref{xu}. Also, note that the system above is no longer Hamiltonian since the introduction of the $u_i$'s is not a symplectic change of variables. The upside is that it is now suitable for Painlev\'e analysis because the nonlinearity is quadratic and therefore formal series solutions will converge in a sufficiently small punctured disc (see section \ref{conclusions} for details). System \eqref{xu} can be regarded as the Ricci soliton analogue of (2.3)-(2.6) in \cite{DW2001}. One important difference is that here we are using the geodesic length $t$ as our independent variable, whereas in the Einstein case a new independent variable $s$ that satisfies $ds=e^{-\dd \cdot \qq/2}dt$ was used. In the Ricci soliton case this change of variables actually makes the system more complicated because it involves inverting an exponential integral function, which might destroy the Painlev\'e property of the resulting system.  \\

\subsection{The Berard Bergery ansatz}
The Hamiltonian formulation holds for manifolds where the hypersurface is not a warped product. One of the simplest examples of this is the B\'erard Bergery ansatz, where the hypersurface is the total space $P$ of an $S^1$-bundle $\pi: P \rightarrow B$ over a Fano K\"ahler-Einstein base $(B,g_B)$. Cohomogeneity one Einstein metrics generalizing the Page metric \cite{Pa1978} were constructed in such bundles by B\'erard Bergery in \cite{BB1982} (these metrics are not K\"ahler but conformally K\"ahler). K\"ahler-Ricci solitons have been also constructed in this context by Koiso and Cao \cite{Koi1990, Cao1996}, and Feldman, Ilmanen, and Knopf \cite{FIK2003} (see also \cite{Wink2017, Sto2017, Appleton2017} for recent developments on non-K\"ahler solitons on these bundles). In this case the extended dimension vector is $\dd=(1,d_2,-2)$ with $d_2$ even and the metric on $P$ has the form
$$ \bar{g}=dt^2+g_t=dt^2+e^{q_1} \theta \otimes \theta + e^{q_2} \pi^* g_B, $$
where $\theta$ is a principal $U(1)$ connection on $P$ (cf. \cite{PP1987,DW2011}). Using equation (1.3) in \cite{WZ1986} the scalar curvature of $g_t$ is given by
$$ S=A_2\, e^{-e_{q_2}}+A_3\, e^{q_1-2q_2},$$
where $A_2>0$ and $A_3<0$ are constants. Thus, $\mathcal{W}$ is composed of the vectors $\ww_1=(0,-1,0)$, $\ww_2=(1,-2,0)$. As before, we can introduce the variables $x_1= (A_2/E) e^{\ww_1 \cdot \qq} $, $x_2= (A_3/E)   e^{\ww_2 \cdot \qq}$, and $x_3= E e^{\frac{1}{2}\dd\cdot \qq}$ so that 
$$ U= \left( \begin{array}{ccc} 0 &-1 & 0\\ 1 & -2 & 0 \\ 1/2 & d_2/2 & -1 \end{array}  \right).$$
The corresponding conjugate momenta $y_i$ are defined by \eqref{conjugatemomenta} and $M=UJU^T$ is now given by
$$ M= \left( \begin{array}{ccc}
-1/d_2 & -2/d_2 & 0 \\
-2/d_2 & -1-4/d_2 & 0 \\
0 & 0 & 1/4
\end{array}       \right) .$$
Thus, the Hamiltonian in these variables becomes
$$ H=-\frac{1}{d_2 \, x_3}\big( x_1^2 \, y_1^2+ 4 x_1 x_2 \, y_1 y_2 + (d_2+4) x_2^2 y_2^2 \big)+\frac{x_3 y_3^2}{4}-x_3 (x_1+x_2+1). $$
To simplify the Painlev\'e analysis we introduce a new set of auxiliary variables $v_1:=(x_1y_1+2x_2y_2)/(d_2x_3)$, $v_2:=x_2y_2/x_3$, and $v_3:=y_3/2$, so that the canonical equations associated to the Hamiltonian above become
\begin{equation} \label{xv}
\begin{array}{llc}
\dot{x}_1  =-2x_1 v_1 & \quad \dot{v}_1 =-v_1 v_{3}+\displaystyle\frac{1}{d_2}( x_1+2x_2) \\[3mm]
\dot{x}_2  =-2x_2 (2v_1+v_2) & \quad \dot{v}_2 =-v_2 v_{3}+ x_2 \\[3mm]
\dot{x}_{3} = x_{3} \, u_{3} & \quad \dot{v}_{3} =-d_2 v_1^2-v_2^2. \\
\end{array}
\end{equation}
In these variables the $H=0$ condition becomes 
$$H=-x_3(d_2v_1^2+v_2^2-v_3^3+x_1+x_2+1)=0.$$

We proceed to carry out a Painlev\'e analysis of \eqref{xu} first. System \eqref{xv} will be analyzed in section \ref{berard}. We begin with the simpler case $r=1$ since many of its features will be useful to study the more general case.

\section{Products with one factor} \label{onefactor}
We begin the analysis of system \eqref{xu} in the case where there is only one factor, i.e. $r=1$. As mentioned in the introduction, the first step is to assume the equations are complex and find the leading order terms of the ansatz expansions. We take expansions around $t_0=0$ of the form
\begin{align*}
x_1=a_0 t^\alpha +\ldots  & \qquad u_1= c_0 t^\gamma +\ldots \\ x_2=b_0 t^\beta+\ldots & \qquad u_2= e_0 t^\epsilon +\ldots ,
\end{align*}
where we are assuming $t$ to be a new complex variable that is not related to the parameter used in the previous section (so $t$ has no geometric meaning). Recall that the position of the singularity at $t_0=0$ is arbitrary and actually constitutes a free parameter of the system. Substituting in \eqref{xu} we see that the terms which are potentially leading on each side of the corresponding equation are given by 
\begin{align}
a_0 \alpha \, t^{\alpha -1} \, &: \, -2 a_0 c_0 \, t^{\alpha + \gamma} \label{a} \\
b_0 \beta \, t^{\beta -1} \, &: \, b_0 e_0 \, t^{\beta + \epsilon} \label{b}\\
c_0 \gamma \, t^{\gamma -1 } \, &: \, -c_0 e_0 \, t^{\gamma + \epsilon}, \,\, d_1^{-1}a_0 \, t^{\alpha} \label{c}\\
e_0 \epsilon \, t^{\epsilon -1 } \, &: \, -d_1 c_0^2 \, t^{2\gamma}. \label{d}
\end{align}
The following proposition shows the admissible values of both the leading exponents and coefficients.
\begin{pro}
The leading order terms satisfy $\gamma=-1$, $\epsilon=-1$, and one of the following two possibilities:
\begin{equation}\label{uno}\begin{array}{llll}
a_0= d_1(d_1-1) & c_0 = 1 & \alpha=-2 \\
b_0=\textnormal{Arbitrary} & e_0=d_1 & \beta = d_1,
\end{array}
\end{equation}
or
\begin{equation}\label{dos}\begin{array}{llll}
a_0=\textnormal{Arbitrary} & c_0 = \pm 1/\sqrt d_1 & \alpha=\mp 2/ \sqrt d_1 \\
b_0=\textnormal{Arbitrary} & e_0=1 & \beta = 1.
\end{array}
\end{equation}
\end{pro}
\begin{proof}
First we establish that $\alpha \neq 0$. By contradiction, if $\alpha=0$ then \eqref{a} shows that $\gamma>-1$ (otherwise the negative exponent on the right hand side can not be matched by the left). Then, $\epsilon>-1$ by \eqref{d} and using \eqref{b} we see that $\beta=0$. Using \eqref{c} and \eqref{d} it is easy to deduce that $\gamma, \epsilon \geq 0$ and therefore the expansions are not singular. \\
\indent Given $\alpha \neq 0$, we deduce $\gamma=-1$ and $\epsilon=-1$ from \eqref{a} and \eqref{d} respectively. From \eqref{c} we see that $\alpha \geq -2$. If $\alpha=-2$, it is easy to see that the rest of the constants are given by \eqref{uno}. If $\alpha >-2$, \eqref{c} implies $e_0=1$; subsequently \eqref{d} yields $c_0=\pm1/\sqrt{d_1}$. Finally, $\alpha=\mp 2/\sqrt{d_1}$ and $\beta=d_1$ follow easily.
\end{proof}
In what follows we will have to analyze the series corresponding to the two different possibilities singled out in the previous proposition. These two cases can be distinguished by the value of the exponent $\alpha$. In the next section, when we analyze the system for $r>1$, the leading order terms will be classified by the leading exponent of $x_1$ in an analogous way.
\begin{rem}
Already at this stage we can see that in the case where $\alpha>-2$ we need $d_1$ to be perfect square in order to obtain series on a rational power of $t$. Also, notice that $a_0$ (in the $\alpha >-2$ case) and $b_0$ are arbitrary and therefore can be regarded as free parameters of the system. The arbitrariness of $b_0$ comes from the fact that the soliton potential is determined up to a constant. In particular, by looking at the definition of $x_2$ we see that a translation of the potential $u\mapsto u+c$ for some constant $c$ implies that $b_0\mapsto e^{-c}\, b_0$.
\end{rem}
\begin{rem}
It is possible to consider solutions which have poles at infinity. Such solutions have poles at $s=0$ for the new variable $s=1/t$. It is easy to show no such expansions exist. 
\end{rem}

\subsection{The $\alpha=-2$ case}
We now proceed to find the resonances for system \eqref{xu} using the leading order terms given in \eqref{uno}. Recall that resonances are the steps in the expansion at which free parameters may enter. To this end we take series of the form
\begin{align*}
x_1=\sum_{i=0}^\infty a_i t^{-2+i}  & \qquad u_1= \sum_{i=0}^\infty c_i t^{-1+i} \\  x_2=\sum_{i=0}^\infty b_i t^{d_1+i} & \qquad u_2= \sum_{i=0}^\infty e_i t^{-1+i}.
\end{align*}
Substituting in \eqref{xu} and reordering terms with the same exponents yields the following recursion relation for $i\geq 1$:
\begin{equation}\label{recr1}
\left(\begin{array}{cccc}
i & 0 & 2d_1(d_1-1) & 0 \\
0 & i & 0 & -b_0 \\
-1/d_1 & 0 & i+(d_1-1) & 1 \\
0 & 0 & 2d_1 & i-1
\end{array} \right) 
\left( \! \! \begin{array}{c}
 a_i \\ b_i \\ c_i \\ e_i \end{array} \! \! \right)=
\left( \begin{array}{c}
-2 \sum_{k=1}^{i-1} a_{i-k} c_k \\  \sum_{k=1}^{i-1} b_{i-k} e_k \\ - \sum_{k=1}^{i-1} c_{i-k} e_k \\ -d_1  \sum_{k=1}^{i-1} c_{i-k} c_k 
\end{array} \right).
\end{equation}
We denote the matrix on the left hand side of \eqref{recr1} by $X(i)$ and the vector by $\vv_i$. We have resonances whenever $\det X(i)=0$. Remember that we are interested in finding out the cases where the resonances are either integers or rational. A simple computation shows that this determinant can be factorized as 
$$ \det X(i)= i(i+d_1-1)(i+1)(i-2), $$
so that resonances occur at $i=-d_1+1, -1, 0,$ and $2$. Thus, resonances are always integers regardless of the value of $b_0$ or $d_1$. The first is negative and strictly less than -1 when $d_1\geq 3$. It has no geometric interpretation. The resonance at $i=-1$ corresponds to the arbitrariness of the initial parameter $t_0$, and the resonance at $i=0$ corresponds to the arbitrariness of $b_0$. In order to get a formal series solution we need to verify that compatibility conditions are met at the top resonance $i=2$. 

\begin{pro}\label{r=1}
Compatibility conditions hold at the top resonance $i=2$ and therefore the corresponding ansatz yields a formal a series solution of \eqref{xu}. Moreover, exactly one free parameter enters the expansion at this step. 
\end{pro}
\begin{proof}
It suffices to compute the first two steps of the recursion. Notice that at step $i=1$ the right hand side of \eqref{recr1} vanishes, so $\vv_1=0$. Likewise, the right hand side vanishes at step $i=2$, which readily implies that compatibility conditions are met at the resonance. To see that one free parameter enters the expansion at the top resonance notice that $\rank X(2)=3$. In particular $\vv_2$ can be set to be any element of $\ker X(2)$. This kernel is spanned by $(d_1-1,b_0,-1/d_1,2)$. 
\end{proof}

Recall from the previous section that only solutions where $H=0$ correspond to smooth metrics. It is interesting to note the free parameter at the resonance $i=2$ can be chosen so that the resulting formal solution lies on this level set. To see this, recall the condition $H=0$ is equivalent to $d_1 u_1^2-u_2^2+x_1+1=0$. The coefficients for $t^{k}$ in this expression are given by\\

$t^{-2}: d_1c_0^2-e_0^2+a_0=0. $\\

$t^{-1}: 2d_1c_0c_1-2e_0e_1+a_1=0.$\\

$t^{0}: d_1(2c_0c_2+c_1^2)-(2e_0e_2+e_1^2)+a_2+1.$\\

To get a solution on $H=0$ the last coefficient must be zero. It is easy to see that taking $\vv_2=\lambda (d_1-1,b_0,-1/d_1,2)$ with $\lambda =1/3(d_1+1)$ makes the coefficient of $t^{0}$ vanish.\\

The proposition above gives a 3-parameter family of formal series solutions of \eqref{xu} in the case $r=1$. The general solution of this system has four free parameters so this family does not pass the full Painlev\'e test. In the context of Ricci solitons, the free parameter that we get at the top resonance is lost by restricting the solution to $H=0$, so we have a 2-parameter family of solutions which correspond to steady Ricci solitons (although this family actually defines the same soliton since the free parameters are $t_0$ and $b_0$, and we have seen that these variables have no geometric meaning).

\subsection{The $\alpha= \mp 2/\sqrt{d_1}$ case}
Now we analyze the resonances for the expansion where the leading exponent of $x_1$ is given by $\alpha=\mp 2/\sqrt{d_1}$. As stated above, to obtain expansions on a rational power of $t$ we will require $d_1$ to be a perfect square. The rest of the leading terms are given in \eqref{dos}. Thus, we will consider series of the form
\begin{align*}
x_1=\sum_{i=0}^\infty a_i t^{\alpha+iQ}  & \qquad u_1= \sum_{i=0}^\infty c_i t^{-1+iQ}  \\  x_2=\sum_{i=0}^\infty b_i t^{1+iQ} & \qquad u_2= \sum_{i=0}^\infty e_i t^{-1+iQ},
\end{align*}
where $Q=2/\sqrt{d_1}$. Substituting into \eqref{xu} and collecting terms with the same exponents we arrive at the recursion relation: 
\begin{equation}\label{recr2}
\left(\begin{array}{cccc}
iQ & 0 & 2a_0 & 0 \\
0 & iQ & 0 & -b_0 \\
0 & 0 & iQ & \pm 1/\sqrt{d_1} \\
0 & 0 & \pm 2\sqrt{d_1} & iQ-1
\end{array} \right) 
\left( \! \! \begin{array}{c}
 a_i \\ b_i \\ c_i \\ e_i \end{array} \! \! \right)=
\left( \begin{array}{c}
-2 \sum_{k=1}^{i-1} a_{i-k} c_k \\  \sum_{k=1}^{i-1} b_{i-k} e_k \\ - \sum_{k=1}^{i-1} c_{i-k} e_k + d_1^{-1}a_{i-\sqrt{d_1}\pm 1}\\ -d_1  \sum_{k=1}^{i-1} c_{i-k} c_k 
\end{array} \right)
\end{equation}
(we adopt the convention that $a_k=0$ when $k<0$). To simplify the notation in what follows it will be useful to set $\iota=iQ$. As before, denote the matrix on the left by $X(\iota)$ and the column vector by $\vv_i$. We will get resonances at the values of $i$ where $\det X(\iota)$ vanishes. It is easy to see that $\det X(\iota)$ factorizes as
$$ \det X(\iota)=\iota^2 (\iota+1)(\iota-2), $$
so we have resonances at $\iota=-1,0,0,2$. Notice that resonances are integers and do not depend on the value of the free parameters $a_0,b_0$, the dimension $d_1$, or the sign of $\alpha$. The first three resonances correspond to the free parameters $t_0, a_0$, and $b_0$, respectively. To pass the (weak) Painlev\'e test we need to verify that compatibility conditions hold at the top resonance $\iota=2$.
\begin{pro}
Compatibility conditions for \eqref{recr2} hold at the top resonance $\iota=2$. Moreover, exactly one free parameter enters the expansion at this resonance. This free parameter can be chosen so that the formal series solution lies on $H=0$. 
\end{pro}
\begin{proof}
As before, we compute the first steps of the recursion. The right hand side of \eqref{recr2} behaves differently depending on the sign of $\alpha$ so we consider both cases separately. We do the case $\alpha >0$ first (so we get minus signs in $X(\iota)$). At step $i=1$, the right hand side vanishes (we are using the convention $a_k=0$ if $k<0$) and so $\vv_1=0$. By induction, $\vv_i=0$ for $1 \leq i <  \sqrt{d_1}$. At the resonance $i=\sqrt{d_1}$ the right hand vanishes so compatibility conditions are satisfied. We have $\rank X(2)=3$ so one free parameter enters the expansion. The kernel $\ker X(2)$ is spanned by $(-a_0/\sqrt{d_1},b_0, 1/\sqrt{d_1},2)$. Multiplying the latter vector by $1/6$ gives us a value for $\vv_{\sqrt{d_1}}$ such that the solution lies on $H=0$. \\

If we choose $\alpha <0$ the analysis is slightly more complicated because we need to consider two different subcases. First, assume that $d_1 > 4$. As before, the right hand side of the recursion vanishes at $i=1$, so $\vv_1=0$. By induction $\vv_i=0$ for $1 \leq i <\sqrt{d_1}-1$. At $i=\sqrt{d_1}-1$ the right hand side of \eqref{recr2} ceases to vanish (the third entry is non zero), so $\vv_{\sqrt{d_1}-1}\neq 0$. When $i=\sqrt{d_1}$ we have a resonance where the right hand side of the recursion vanishes (since all the nonzero coefficients in the sums appear multiplied by zero). Thus, compatibility conditions are satisfied. Again, $\rank X(2)=3$ and $\ker X(2)$ is spanned by $(a_0/\sqrt{d_1},b_0,-1/\sqrt{d_1},2)$. Multiplying this vector by 1/6 yields a solution where $H=0$.\\

If $d_1=4$ the right hand side of \eqref{recr2} at step $i=1$ is $(0,0,a_0/4,0)$ and one must solve the system in order to get $\vv_1=(0,a_0 b_0/2,0,a_0/2)$. The right hand side at the resonance $i=2$ is given by $\left(0, a_0^2\,b_0/4 ,0 ,0  \right)$. This vector is in the image of $X(2)$ and therefore compatibility conditions are satisfied. Indeed, $\rank X(2)=3$ and solutions of the system at step $i=2$ are of the form
\begin{equation}\label{alexakis}
 \left( 0, \frac{a_0^2 \, b_0}{8},0,0  \right) + \lambda \bigg( a_0, 2b_0, -1,4 \bigg).
\end{equation}
for $\lambda \in \mathbb{R}$. By taking $\lambda=(4-a_0^2)/48$ we obtain solutions that lie on $H=0$. 
\end{proof}
This proposition shows that expansions corresponding to the choice of leading order terms \eqref{dos} yield formal solutions of \eqref{xu} regardless of the sign of $\alpha$. Moreover, these solutions have 4 free parameters: $t_0$, $a_0$, $b_0$, and the free parameter entering at the top resonance. Assuming they converge, the series in this family satisfy the weak Painlev\'e property for every perfect square $d_1$. Additionally, $Q=1$ when $d_1=4$ so the system passes the strong Painlev\'e test in this case. In terms of Ricci solitons, (after taking into account that $t_0$ and $b_0$ are geometrically meaningless), these solutions correspond to a 1-parameter family of steady Ricci solitons.
\begin{rem}
We have yet to attach any geometric meaning to the free parameter $a_0$. We will go back to this issue in further detail in Section \ref{conclusions}.
\end{rem}

\section{Products with multiple factors} \label{arbfactor}
We now turn our attention to warped products with multiple factors, that is, when $r\geq 2$. The analysis in this case retains many of the features from the $r=1$ setting, however the presence of more factors means there is a wider choice of possible leading order terms. We begin by taking expansions of the form
\begin{equation}\label{lead}
\begin{aligned}
x_i=a_{i,0} t^{\alpha_i} +\ldots  & \qquad u_i= c_{i,0} t^{\gamma_i} +\ldots \quad \text{for } 1\leq i\leq r  \\ x_{r+1}=b_0 t^{\beta}+\ldots  & \qquad u_{r+1}= e_0  t^{\epsilon} +\ldots ,
\end{aligned}
\end{equation}
and substitute in \eqref{xu}. The possible leading terms are given respectively by
\begin{subequations} 
\begin{align}
a_{i,0} \alpha_i t^{\alpha_i -1} \, &: \, -2 a_{i,0} c_{i,0} t^{\alpha_i + \gamma_i}  \label{aa}\\
b_0 \beta t^{\beta -1} \, &: \, b_0 e_0 t^{\beta + \epsilon} \label{bb} \\
c_{i,0} \gamma_i t^{\gamma_i -1 } \, &: \, -c_{i,0} e_0 t^{\gamma_i + \epsilon}, \,\, d_i^{-1}a_{i,0} t^{\alpha_i}  \label{cc}\\
e_0 \epsilon t^{\epsilon -1 } \, &: \, -d_1 c_{1,0}^2t^{2\gamma_1} , \, -d_2 c_{2,0}^2t^{2\gamma_2} ,\ldots, \, -d_r c_{r,0}^2t^{2\gamma_r}.\label{dd}
\end{align}
\end{subequations}
When we have more than one factor the behaviour of the leading order terms is much more complicated because \eqref{dd} has various different terms which could potentially be leading, so the relationship between $\gamma_i$ and $\epsilon$ is not as straightforward as in the $r=1$ case. The following lemma is useful to identify the possible leading order terms.
\begin{lem} \label{leadorder}
The leading order terms satisfy the following properties:
\begin{enumerate}
 \item There exists at least one $i$ such that $\alpha_i\neq 0$.
 \item $\gamma_i\geq -1$ for all $i$.
 \item $\alpha_i\geq -2$ for all $i$.
 \item $\epsilon=-1$ always.
 \item $\beta=e_0$ always.
 \item If $\alpha_i \neq 0$ then $\gamma_i=-1$. 
 \item If $\alpha_i = 0$ then $\gamma_i=1$.
 \item If $\alpha_i=-2$ for some $i$, then either $\alpha_j = 0$ or $\alpha_j=-2$ for all $1\leq j \leq r$. Likewise, if $\alpha_i>-2$ and $\alpha_i\neq 0$ for some $i$, then either $\alpha_j = 0$, or $\alpha_j >-2$, $\alpha_j\neq 0$ for all $1\leq j \leq r$.
\end{enumerate}
\end{lem}
\begin{proof}
(1) Assume by contradiction that $\alpha_i=0$ for all $i$. Then \eqref{aa} implies $\gamma_i>-1$ for all $i$; and by \eqref{dd} $\epsilon>-1$ as well. \eqref{bb} gives $\beta=0$. From \eqref{cc} we see that $\gamma_i, \varepsilon\geq 0$ and thus the solutions are not singular. (2) follows directly from \eqref{aa}, otherwise the exponents can not be balanced; this together with \eqref{cc} readily imply (3). (4) comes from noting that \eqref{bb} implies that $\epsilon \geq -1$ (otherwise the exponents can not be balanced). Then, (1) shows there exist terms with leading exponent $-2$ on the right hand side of \eqref{dd}. These terms can not cancel each other out because all their coefficients are strictly negative, so $\epsilon =-1$ to balance both sides. (5) follows immediately from (4) and \eqref{bb}. (6) is directly implied by \eqref{aa}. For (7), notice that $\alpha_i=0$ implies that $\gamma_i \neq 0$, otherwise \eqref{cc} can not be balanced. Likewise, if $\gamma_i>1$ \eqref{cc} can not be balanced. On the other hand, if $\gamma_i<1$ then $\gamma_i>-1$, otherwise \eqref{aa} can not be balanced. From \eqref{cc} we have $\gamma_i=-e_0\leq -1$ a contradiction. Finally to show (8) notice that $\alpha_i=-2$ implies (using \eqref{dd}) that $e_0>1$ whereas $\alpha_i>-2$, $\alpha_i\neq0$ implies $e_0=1$ (using \eqref{cc}).
\end{proof}
The lemma above is useful because it allows us to classify the possible leading order terms in a way that is analogous to the $r=1$ case. In particular, the leading exponents $\alpha_i$ can be organized in two classes (changing the indices if necessary); one consisting of nonzero exponents $\alpha_1 \leq \ldots \leq \alpha_l$, and a second one consisting of vanishing exponents $\alpha_{l+1}= \ldots= \alpha_{r}=0$ (naturally it is possible to have no vanishing leading exponents, i.e. $l=r$). The last statement of the lemma can be rephrased as saying that the nonzero exponents must either all be equal to -2 or otherwise all be strictly greater than -2. Note that the lemma does not say anything about the sign, and in the latter case some of the exponents may actually be positive. As in the $r=1$ case, we will leave the geometric interpretation of these observations for Section \ref{conclusions}.\\

Following the lemma above, we can reduce the Painlev\'e analysis of a multiple warped product to two cases: one where $\alpha_1=-2$, and one where $\alpha_1>-2$ and $\alpha_1\neq 0$. We will study each of them separately. As above, in what follows we  assume that $\alpha_i\neq 0$ for $1\leq i \leq l$ and $\alpha_i=0$ for $l+1 \leq i \leq r$.

\subsection{Case $\alpha_1=-2$}\label{I}
We begin by identifying the rest of the leading order terms in the expansion \eqref{lead}. As stated above, in this case we have $\alpha_1= \ldots =\alpha_l=-2$ and $\alpha_{l+1}=\ldots=\alpha_r=0$.
\begin{pro}\label{coefI}
If $\alpha_1=-2$, the leading order terms of the corresponding series expansion are given by Lemma \ref{leadorder} and the following relations:
\begin{equation} 
\begin{array}{lll}
a_{i,0}=d_i(e_0-1)\quad & c_{i,0}=1 \quad & \textnormal{if } 1\leq i \leq l \\
a_{i,0}=\textnormal{Arbitrary }  & c_{i,0}=a_{i,0}\,d_i^{-1}(e_0+1)^{-1} & \textnormal{if } l+1 \leq i \leq r \\
b_0=\textnormal{Arbitrary } & e_0=\sum_{k=1}^l d_k
\end{array}
\end{equation}
\end{pro}
\begin{proof}
Note that for $1 \leq i \leq l$, \eqref{aa} yields $c_{i,0}=1$. From this and \eqref{dd} we readily deduce that $e_0=-(d_1+\ldots +d_l)$, and $a_{i,0}$ is then determined by \eqref{cc}. When $l+1 \leq i \leq r$ \eqref{aa} does not give any information on $c_{i,0}$, so instead of a precise value, \eqref{cc} provides the relation $d_i(e_0+1)c_{i,0}=a_{i,0}$. We can take $a_{i,0}$ to be arbitrary; this determines $c_{i,0}$. Since $b_0$ only appears in \eqref{bb} it is a free parameter. For the sake of completeness, recall that the lemma above shows that the exponent $\beta=e_0$.
\end{proof}
We now have a full set of initial conditions. Note that already at this stage we have $r-l+1$ free parameters in the expansion, namely $b_0$ and $a_{i,0}$ for $l+1 \leq i \leq r$. To compute the full recursion relation associated to these leading order terms we take expansions of the form
\begin{align*}
&x_{j}=\sum_{i=0}^\infty a_{j,i} t^{-2+i},    && u_j= \sum_{i=0}^\infty c_{j,i} t^{-1+i} \qquad \text{for } 1\leq j\leq l  \\ 
&x_{j}=\sum_{i=0}^\infty a_{j,i} t^{i}  ,  && u_j= \sum_{i=0}^\infty c_{j,i} t^{1+i} \quad \text{for } l+1\leq j\leq r  \\
&x_{r+1}=\sum_{i=0}^\infty b_i t^{\beta+i} ,   &&u_{r+1}= \sum_{i=0}^\infty e_i  t^{-1+i} 
\end{align*}
and substitute in \eqref{xu}. To make the resulting expression more manageable we introduce the following notation.
\begin{notation}
We denote $\aaa_i=(a_{1,i}, \ldots, a_{r,i})$. Thus, the entries of this vector are the coefficients of $t^{\alpha_j+i}$ in the expansion of each $x_j$. Since we often have to consider the vanishing and non-vanishing exponents separately, we further decompose this vector as $\aaa_{i}=(\aaa'_i,\aaa''_i)$, where $\aaa'_i=(a_{1,i},\ldots, a_{l,i})$ and $\aaa''_{i}=(a_{l+1,i},\ldots, a_{r,i})$. We make an analogous convention for $\cc_i$ and decompose it as $\cc_i=(\cc_i',\cc_i'')$. We also denote $d=(d',d'')$, where $d'=(d_1,\ldots, d_l)$ and $d''=(d_{l+1},\ldots, d_r)$. Whenever we have an operation between two vectors we will think of it as being performed coordinatewise, so for example $\aaa'_i \cc'_j=(a_{1,i}\, c_{1,j},\ldots , a_{l,i}\, c_{l,j})$. Exponentiation is also meant to be understood componentwise, so for example $d'^{-1}=(d_1^{-1},\ldots, d_l^{-1})$. If there is a product between a vector and a scalar we use the usual multiplication of a vector by a scalar factor, so for example a term of the form $\cc_{i}' \, e_{j}$ is just a  rescaling of the vector $\cc_i'$. Finally, the $n\times n$ identity matrix is denoted by $I_n$. To make visualization easier, one may think of the $r=2$, $l=1$ case so $\aaa_i'=a_{1,i}$ and $\aaa_i''=a_{2,i}$ in what follows.
\end{notation}
With this notation the recursion relation can be written as

\begin{equation}\label{recr3}
\begin{aligned}
 \left(  \begin{array}{ccc|ccc}
iI_{l} & & & 2\diag( \aaa_0') & & \\
 & iI_{r-l} & & & 0 & \\
 & & i & & & -b_0 \\ \hline
-\diag(d'^{-1}) & &  &(i+(e_0-1))I_l & & \cc_0' \\
 & -\diag(d''^{-1}) & & & (i+(e_0+1))I_{r-l} & \cc_{0}'' \\
 & & 0 & 2 d'  & & i-1 \\
\end{array} \right)  &
\left(  \begin{array}{c}
\aaa_{i}' \\
\aaa_{i}''\\
b_i\\ \hline
\cc_i' \\
\cc_{i}''\\
e_i
\end{array} \right)=\\
 \left( \begin{array}{c}
-2\sum_{k=1}^{i-1} \aaa_{i-k}' \, \cc_{k}' \\
-2\sum_{k=0}^{i-1} \aaa_{i-2-k}'' \, \cc_{k}''\\
\sum_{k=1}^{i-1} b_{i-k} \, e_k\\ \hline
-\sum_{k=1}^{i-1} \cc_{i-k}' \, e_{k} \\
-\sum_{k=1}^{i-1} \cc_{i-k}'' \, e_{k} \\
-\langle d' ,\sum_{k=1}^{i-1}\cc_{i-k}' \, \cc_k'  \rangle -\langle d'' ,\sum_{k=0}^{i-4}\cc_{i-4-k}'' \, \cc_k''  \rangle \end{array} \right) &
\end{aligned}
\end{equation}

Before we compute the resonances we make a few remarks. First, the last entry on the right hand side should be understood as being the standard (real) inner product of the corresponding vectors. Also, notice that $\cc_0'=(1,\ldots, 1)$ due to Proposition \ref{coefI}. As before, we denote the $(2r+2)\times (2r+2)$ matrix on the left hand side by $X(i)$ and the column vector by $\vv_i$. We proceed to compute the resonances.

\begin{pro}
The determinant of $X(i)$ factorizes as
$$ \det X(i)=i^{r-l+1}(i+1)(i-2)(i+(e_0+1))^{r-l}(i+(e_0-1)) p(i)^{l-1}, $$
where $p(i)= i^2+i(e_0-1)+2(e_0-1)$. 
\end{pro}
\begin{proof}
If we divide $X(i)$ into four $(r+1) \times (r+1)$ blocks $$X(i)=\left( \begin{array}{cc} A & B \\ C & D \end{array} \right)$$
it is possible to use that $\det X(i)=\det (AD-BC)$ since $[ C,D ]=0$. One can use row and column operations to further rearrange $AD-BC$ into a block upper triangular matrix and factorize $( i (i+(e_0+1))^{r-l}$ from the determinant. The remaining block is given by the $(l+1) \times (l+1)$ matrix
$$ \left(  \begin{array}{ccc|c}
p(i) & & & i \\
 & \ddots & & \vdots \\
 & & p(i) & i\\ \hline
2id_1 & \ldots & 2i d_{l} & i(i-1)
\end{array} \right). $$
The proposition follows by induction if we compute the latter determinant along the last column.
\end{proof}
Notice that if $2\leq e_0 \leq 8$, $p(i)$ has no real roots, otherwise its roots are negative, and therefore meaningless from the point of view of Painlev\'e analysis. Likewise, the roots $i=-e_0-1$ and $i=-e_0+1$ (of multiplicities $r-l$ and 1, respectively) do not have any free parameters associated to them. Therefore, the only steps at which free parameters enter the expansion are $i=-1, 0$, and $2$. The resonance at $i=-1$ corresponds to the arbitrariness of $t_0$, whereas at $i=0$ we have the free parameters $b_0$ and $\aaa_0''$. This matches the multiplicities of these roots, which are $1$ and $r-l+1$, respectively. In the next proposition we show that compatibility conditions are satisfied at the top resonance $i=2$.
\begin{pro}
Compatibility conditions are satisfied at the top resonance $i=2$ for any choice of free parameters $a_{i,0}$, $l+1\leq i \leq r$. Furthermore, the free parameter entering the expansion at this step can be chosen so that the formal solution lies on the level set $H=0$.
\end{pro}
\begin{proof}
We compute the first two steps of the recursion. At step $i=1$ the right hand side of \eqref{recr3} vanishes, so $\vv_1=0$. At $i=2$ the matrix $X(2)$ has rank $2r+1$ and the right hand side is given by
$$ \left(0,-2 \, \aaa_0'' \, \cc_0'', 0\,\vert \, 0, 0,0 \right).$$
It is easy to see that the vector or the right is in the image of $X(2)$, so compatibility conditions are satisfied at the top resonance. The general solution is given by
\begin{align*}
&a_{i,2}=\lambda \frac{d_i (e_0-1)}{e_0} && c_{i,2}=- \lambda \frac{1}{e_0} &&& \textnormal{for } 1\leq i \leq l \\[2mm]
&a_{i,2}= -a_{i,0} \, c_{i,0} && c_{i,2}= -\frac{a_{i,0} \, c_{i,0}}{d_i (e_0+3)}- \lambda \frac{c_{i,0}}{e_0+3} &&&  \textnormal{for } l+1\leq i \leq r\\
& b_2= \lambda b_0 && e_2=2 \lambda,
\end{align*}
where $\lambda \in \mathbb{R}$ is a free parameter. The condition $H=0$ in this case is equivalent to 
$$ \sum_{k=1}^{r}d_k u_k^2-u_{r+1}^2+\sum_{k=1}^{r}x_k+1=0.$$
The coefficients for the various powers of $t$ in this expression are given by\\

$t^{-2}$ : $\sum_{k=1}^l d_k c_{k,0}^2- e_0^2 + \sum_{k=1}^l a_{k,0}=0$\\

$t^{-1}$ : $\sum_{k=1}^l 2d_k c_{k,0}c_{k,1}- 2e_0 e_1 + \sum_{k=1}^l a_{k,1}=0$\\

$t^0$ : $\sum_{k=1}^l d_k (2c_{k,0}c_{k,2}+c_{k,1}^2) -2(e_0e_2+e_1^2) +\sum_{k=1}^l a_{k,2}+ \sum_{k=l+1}^r a_{k,0}+1.$\\

Substituting with the values just obtained shows that if we take $$\lambda=\frac{1+\sum_{k=l+1}^r a_{k,0}}{3(e_0+1)}$$
the expression for $t^0$ vanishes. 
\end{proof}

Using the proposition above we have obtained a $r-l+3$-parameter family of formal solutions of \eqref{xu} regardless of the dimensions of the factors. Since the general solution of this system has $2r+2$ free parameters, the expansions corresponding to the $\alpha_1=-2$ case do not pass the full Painlev\'e test. In the context of Ricci solitons, the restriction to $H=0$ looses the free parameter at the top resonance. Accounting for the arbitrariness of $t_0$ and $b_0$ we obtain an $r-l$-parameter family of formal solutions corresponding to steady Ricci solitons.

\subsection{Case $\alpha_1>-2$}\label{II}
We now analyze the case when the first $l$ exponents $\alpha_1 \leq \ldots \leq \alpha_l$ are nonzero and strictly greater than -2 (recall that some of them may be positive). The remaining exponents $\alpha_{l+1}, \ldots, \alpha_r$ vanish as before.  Proposition \ref{leadorder} determines the value of each $\gamma_i$, however this time it does not yield an explicit value for the nonzero $\alpha_i$'s. The following proposition identifies the possible leading order terms.
\begin{pro}\label{coefII} 
The leading order terms corresponding to a series expansion with $\alpha_1>-2$ are obtained by setting
\begin{equation} 
\begin{array}{lll}
a_{i,0}=\textnormal{Arbitrary}\quad & c_{i,0}=-\alpha_i/2 \quad & \textnormal{if } 1\leq i \leq l \\
a_{i,0}=\textnormal{Arbitrary}  & c_{i,0}=a_{i,0}/2d_i & \textnormal{if } l+1 \leq i \leq r \\
b_0=\textnormal{Arbitrary } & e_0=1,
\end{array}
\end{equation}
where the leading exponents $\alpha_i$, $1\leq i \leq l$ can be taken to be any rational point (with non-zero entries) on the ellipsoid $d_1\alpha_1^2 + \ldots + d_l\alpha_l^2=4$ in $\mathbb{R}^{l+1}$.
\end{pro}
\begin{proof}
If $1 \leq i \leq l$, \eqref{aa} yields $c_{i,0}=-\alpha_i/2$. Since $\alpha_i>-2$, this exponent is no longer leading in \eqref{cc}, from which we deduce that $e_0=1$. From this we see that for $l+1 \leq i \leq r$, \eqref{cc} gives $2c_{i,0}d_i=a_{i,0}$. The coefficients $a_{i,0}$ are arbitrary since they only appear in \eqref{cc}. Finally, \eqref{dd} implies that $1=d_1c_{1,0}^2+\ldots+d_lc_{l,0}^2$.
\end{proof}
We now proceed to compute the full recursion relation. In what follows it will be useful to have a common leading exponent for all the singular $x_j$'s, so we will start the corresponding series at level $t^{-2}$ (as opposed to $t^{\alpha_j}$). We call the greatest common divisor $Q:=\textnormal{gcd}\{\alpha_1 ,\ldots, \alpha_l, 2 \}$ and we write $\alpha_j=-2+s_jQ$. Thus, we will be looking for solutions of the form
\begin{equation} \label{series}
\begin{aligned}
&x_{j}=\sum_{i=0}^\infty a_{j,i-s_j}\, t^{-2+iQ},    && u_j= \sum_{i=0}^\infty c_{j,i} \, t^{-1+iQ} \qquad \text{for } 1\leq j\leq l  \\ 
&x_{j}=\sum_{i=0}^\infty a_{j,i}\, t^{iQ}  ,  && u_j= \sum_{i=0}^\infty c_{j,i} \, t^{1+iQ} \quad \text{for } l+1\leq j\leq r  \\
&x_{r+1}=\sum_{i=0}^\infty b_i \, t^{1+iQ} ,   &&u_{r+1}= \sum_{i=0}^\infty e_i  \, t^{-1+iQ}, 
\end{aligned}
\end{equation}
where the convention that coefficients $a_{j,k}=0$ when $k<0$ still holds. We again use the notation $\iota=iQ$. Formal series solutions of \eqref{xu} have to satisfy the recursion relation
\begin{equation} \label{recr4}
\begin{aligned}
 \left(  \begin{array}{ccc|ccc}
\iota I_{l} & & & 2\diag( \aaa_0') & & \\
 & \iota I_{r-l} & & & 0 & \\
 & & \iota & & & -b_0 \\ \hline
0 & & &\iota I_l & & \cc_{0}' \\
 & -\diag(d''^{-1}) & & & (\iota +2)I_{r-l} & \cc_{0}'' \\
 & & 0 & 2 d' \cc_0' & & \iota-1 \\
\end{array} \right) &
\left( \begin{array}{c}
\aaa_{i}' \\
\aaa_{i}''\\
b_i\\ \hline
\cc_{i}' \\
\cc_{i}''\\
e_i
\end{array} \right)=\\
\left( \begin{array}{c}
-2\sum_{k=1}^{i-1} \aaa_{i-k}'\, \cc_{k}' \\[1mm]
-2\sum_{k=0}^{i-2Q^{-1}} \aaa_{i-2Q^{-1}-k}''\, \cc_{k}''\\[1mm]
\sum_{k=1}^{i-1} b_{i-k}e_k\\ \hline
-\sum_{k=1}^{i-1} \cc_{i-k}'e_{k} +(a_{1,i-s_1}/d_1,\ldots,a_{l,i-s_l}/d_l) \\[1mm]
-\sum_{k=1}^{i-1} \cc_{i-k}''e_{k}\\[1mm]
-\langle d' ,\sum_{k=1}^{i-1}\cc_{i-k}' \, \cc_k'  \rangle -\langle d'' ,\sum_{k=0}^{i-4Q^{-1}}\cc_{i-4Q^{-1}-k}'' \, \cc_k''  \rangle
\end{array} \right) & ,
\end{aligned}
\end{equation}

As in the previous section, we call the matrix on the left $X(\iota)$. The following proposition allows us to compute the resonances. 
\begin{pro}
The determinant of $X(\iota)$ factorizes as
$$ \det X(\iota) = \iota^{r+l}(\iota+1)(\iota+2)^{r-l}(\iota-2). $$
\end{pro}
\begin{proof}
If we write $X(\iota)$ in a block form $$X(\iota)=\left( \begin{array}{cc} A & B \\ C & D \end{array} \right)$$
we have that $\det X(\iota)=\det (A)\det(D)$. The first determinant is obviously $\iota^{r+1}$. To compute $\det D$, note that $D$ can be taken to an upper block triangular matrix using row and column operations. From this we can factorise $(\iota+2)^{r-l}$ from the determinant. The remaining block is given by
$$ \left( 
\begin{array}{ccc|c}
\iota & & & c_{1,0} \\
 & \ddots & & \vdots \\
 & & \iota & c_{l,0} \\ \hline
2d_1 c_{1,0} & \ldots & 2d_l c_{l,0} & \iota -1
\end{array}
\right).
$$
The result follows by expanding the determinant of this block along the top row and using induction.
\end{proof}
Thus we get resonances at $\iota=-2,-1,0$, and $2$. As usual, the resonance at $\iota=-2$ does not have any meaning, whereas $\iota=-1$ and $\iota=0$ account for the arbitrariness of $t_0$, and $a_{i,0}$ and $b_0$, respectively. Notice that the resonance at $\iota=0$ has multiplicity $r+l$ but we only have $r+1$ free parameters entering the expansion at this stage (recall that $l\geq 1$ by Proposition \ref{leadorder}). Thus, if $l>1$ there are not enough free parameters entering the expansion at this resonance and therefore \eqref{xu} does not pass either version of the Painlev\'e test in general ($l=1$ still doesn't pass the test due to the resonance at $\iota=-2$). Nevertheless, we will see that compatibility conditions are satisfied at the top resonance $\iota=2$ and therefore there might be expansions that correspond to integrable subsystems.
\begin{pro} \label{ellipsol}
Let $(\alpha_1,\ldots, \alpha_l)$ be a rational point with nonzero entries on the ellipsoid $d_1\alpha_1^2+\ldots+ d_l\alpha_l^2=4$. The associated series expansions \eqref{series} satisfy the compatibility conditions at the top resonance $\iota=2$ for any choice of leading coefficients $a_{i,0}$ and $b_0$. Furthermore, the free parameter at the top resonance can be chosen so that the solution lies on the set $H=0$.
\end{pro}
\begin{proof}
To make computations easier we will assume in what follows that $r=l\geq 2$ (so the double primed terms are not present). The right hand side of \eqref{recr4} can have various different behaviours depending on the dimensions $d_j$ and the leading exponents $\alpha_j$. We identify five different cases.\\
 
The simplest one occurs when $\alpha_1>0$. This implies that $s_j>2Q^{-1}$ for all $1\leq j \leq r$ and thus the right hand side of \eqref{recr4} vanishes at the top resonance $i=2Q^{-1}$. Compatibility conditions are trivially satisfied and the general solution is given by taking $$\vv_{2Q^{-1}}=\lambda  \, (\, \aaa_0' \cc_0'\, , \, b_0 \, \vert- \cc_0' \, , 2) .$$
It is easy to see that setting $\lambda=1/6$ generates a series that remains on $H=0$.\\

Next we consider the case when $-1<\alpha_1<0$. This condition means that $s_j>Q^{-1}$ for all $1\leq j \leq r$. It also implies that the right hand side of \eqref{recr4} vanishes at the resonance $i=2Q^{-1}$ since $s_j+s_k>2Q^{-1}$ for all $1\leq j, k \leq r$ (note that coefficients in the recursion can be nonvanishing only at steps where $i$ is a linear combination of $s_1, \ldots, s_r, 2Q^{-1}$ with integer coefficients). Thus, compatibility conditions are satisfied at the top resonance. The solution is given by taking $\vv_{2Q^{-1}}$ as before. It is possible to verify that terms with negative exponents in $ \sum_{k=1}^{r} d_k u_k^2-u_{r+1}^2+\sum_{k=1}^r x_r +1 $ cancel each other out and only terms with nonnegative exponents remain. Indeed, the coefficient of $t^{-2}$ is easily shown to be zero, whereas the coefficients of $t^{\alpha_j}$ for $\alpha_j<0$ can be shown to vanish using the equations in the bottom right block of $X(\iota)$ (what we called $D$ above). To see this multiply each of the equations corresponding to $\dot{u}_j$, $1\leq j \leq r$ in \eqref{recr4} by $-2 d_j c_{j,0}$ and then add them all to $\iota$ times the bottom  equation. The resulting expression shows that $(\iota+1)e_{i}=a_{j,0}$ when $i=s_j$, which together with the last equation proves the claim. Thus, all we have to do is make sure the coefficient of $t^0$ vanishes. This is achieved by taking $\lambda=1/6$ as before.\\

The assumption $\alpha_1>-1$ constitutes the general case of the recursion since this exponent must satisfy $d_1\alpha_1^2\leq 4$ (with equality holding if $r=1$). If $\alpha_1\leq -1$, then $d_1\leq 4$, so all that is left is to analyze the recursion when the first factor has low dimension. The condition $\alpha_1 \leq -1$ adds some complexity because the right hand side of \eqref{recr4} can cease to vanish at the resonance $\iota =2$. There are two ways in which this can happen:
\begin{enumerate}
	\item If two leading exponents interact with each other, that is, if $s_1+s_j=2Q^{-1}$, $2s_1+s_j=2Q^{-1}$, or $3 s_1+s_j=2Q^{-1}$ for some $j$. Note that these are the only possible combinations that could potentially add up to $2Q^{-1}$ since the leading exponents have to lie on the ellipsoid and thus $\alpha_1>-\sqrt 2$.
	\item If the first leading exponent interacts with itself, that is, if $2s_1=2Q^{-1}$ or $3s_1=2Q^{-1}$. Again, the fact that the exponents lie on the ellipsoid means that these are the only cases we need to consider.
\end{enumerate}

The first kind of interaction occurs only when $r=2$, $\alpha_1=\alpha_2=-1$, and $d_1=d_2=2$. To see this just note the condition $s_1+s_j=2Q^{-1}$ is equivalent to $\alpha_1+\alpha_j=-2$. This, together with the restriction $d_1 \alpha_1^2+ d_j \alpha_j^2 \leq 4$ yields the claim (a similar analysis rules out the other possibilities in (1)). On the other hand, (2) occurs when $\alpha_1=-1$ or $\alpha_1=-4/3$ (note that the latter case holds only if $d_1=2$). We study each of these scenarios individually.\\

First assume $r=2$, $d_1=d_2=2$, and $\alpha_1=\alpha_2=-1$. This implies that $Q=1$ and $s_1=s_2=1$. To check compatibility conditions we need to verify that the vector on the right hand side of \eqref{recr4} is in the image of $X(i)$ when $i=2$. This happens if the right hand side is orthogonal to $\ker X(2)^T$, which is in turn spanned by $(0,0,0,1,1,-1)$. Therefore, we need to check that the inner product
$$ \left( \begin{array}{c}
-c_{1,1}\, e_1+a_{1,1}/2 \\
-c_{2,1}\, e_1+a_{2,1}/2 \\
-2c_{1,1}^2-2c_{2,1}^2
\end{array} \right)\cdot \left( \begin{array}{c}
1 \\
1 \\
-1
\end{array} \right) $$
vanishes. Using the equations corresponding to the block $D$ for $X(1)$ in the recursion we see that $2c_{i,1}^2=-c_{i,1}\, e_1 +a_{i,0} \, c_{i,1}$, which together with the equations from block $A$ and the bottom equation in block $D$ show that the desired product vanishes. \\

The case $r>1$, $\alpha_1=-1$, and $\alpha_2>-1$ can be dealt with in a completely analogous way. This time, the kernel $\ker X(2)^T$ is spanned by $(0,0 \, \vert \, d'\cc_0'\, ,-1)$. Simplifying the expression for each $-2c_{k,s_1}^2$, $1 \leq k \leq r$, using the equations from block $D$ as before, and further use of the first and last equations respectively, we see that the right hand side of the recursion is orthogonal to $\ker X(2)^T$, as desired.\\

The $\alpha_1=-4/3$ case is more complicated because we have to compute two steps in the recursion. Compatibility conditions are satisfied if the right hand side is in the image of $X(2)$, as usual. This happens if and only if
\begin{equation} \label{innerproduct}
 \left( \begin{array}{c} -(c_{1,2s_1} e_{s_1}+c_{1,s_1} e_{2s_1})+ \frac{a_{1,2s_1}}{2} \\
\vdots \\
-(c_{r,2s_1} e_{s_1}+c_{r,s_1} e_{2s_1})\\
-2d_1c_{1,2s_1}c_{1,s_1}-\ldots -2d_rc_{r,2s_1}c_{r,s_1}
\end{array} \right) \cdot
\left( \begin{array}{c}d_1 c_{1,0} \\
\vdots \\
d_r c_{r,0}\\
-1
\end{array} \right)=0.
\end{equation}
In order to verify that this product vanishes it is convenient to rewrite the terms at the bottom using the equations at step $i=s_1$ (equivalently $\iota=2/3$) of the recursion. In particular, for any $1\leq j \leq r$
$$ 2d_jc_{j,2s_1}\,c_{j,s_1}=d_j\,c_{j,2s_1}(2c_{j,s_1} )=-3d_j\, c_{j,2s_1}\, c_{j,0}\,e_{s_1}+\delta_{1,j} \, 3a_{1,0}\, c_{1,2s_1}. $$
With this in mind, the inner product above becomes
$$ \sum_{j=1}^r -4d_j \, c_{j,2s_1}\, c_{j,0}\,e_{s_1} - \sum_{j=1}^r d_j \, c_{j,s_1} \, c_{j,0}\, e_{2s_1} +3a_{1,0}\, c_{1,2s_1}+ \frac{2a_{1,2s_1}}{3}.$$
This expression can be simplified using the relations at step $i=2s_1$ (equivalently $\iota=4/3$) of the recursion. In particular, adding appropriate multiples of the first $r$ equations in block $D$ of $X(4/3)$ we can show that
$$\left( \sum_{j=1}^r -4d_j \, c_{j,2s_1}\, c_{j,0}\right)e_{s_1}= \left( 3e_{2s_1}+\frac{1}{2}e_{s_1}^2-2a_{1,s_1} \right) e_{s_1}. $$
Likewise, using the bottom equation at step $\iota=2/3$ we get
$$- \left( \sum_{j=1}^r d_j \, c_{j,s_1} \, c_{j,0}\right) e_{2s_1}= -\frac{1}{6}e_{2s_1}\, e_{s_1}. $$
Further use of the top equation in block A at step $\iota=4/3$ allows us to get rid of the $a_{1,2s_1}$ term.  The product \eqref{innerproduct} is equivalent to
$$ \left( 3e_{2s_1}+\frac{1}{2}e_{s_1}^2-2a_{1,s_1} \right)e_{s_1}- \frac{1}{6}e_{2s_1}\, e_{s_1}+ 2 a_{1,0}\, c_{1,2s_1}- a_{1,s_1}\, c_{1,s_1} =0.$$
Although seemingly still complicated, the expression above is useful because it only depends on five coefficients. They can be computed directly from the recursion to obtain
\begin{align*}
a_{1,s_1}= -\frac{9}{20}a_{1,0}^2 \qquad c_{1,s_1}&= \frac{3}{20} a_{1,0} \qquad e_{s_1}= \frac{3}{5}a_{1,0} \\
c_{1,2s_1}=-\frac{351}{5600}a_{1,0}^2 \quad & \quad e_{2s_1}= -\frac{243}{700}a_{1,0}^2.
\end{align*}
Substituting these values in the expression above shows that the inner product \eqref{innerproduct} vanishes, as desired.

The arguments above show that compatibility conditions are satisfied in each of the three cases where the right hand side ceases to vanish at the top resonance. To see that the free parameter can be chosen so that the series lies on $H=0$ just note that this parameter comes from adding a multiple of a generator of $\ker X(2)$ to a particular solution of the recursion. An argument similar to the ones above shows that this multiple can be chosen so that the coefficient of $t^0$ vanishes (we do not compute the explicit value for $\lambda$ this time because it depends on the coefficients $a_{j,0}$ in a complicated way).

Finally, notice that if $l\neq r$ the results above still hold since doubly primed factors do not have any effect on the compatibility conditions. This is because the equations for $a_{j,0}$ ($l+1 \leq j \leq r$) are uncoupled from the rest of the system for $i\leq 2Q^{-1}$. As before, we can obtain solutions that lie on $H=0$ by finding an adequate value for the free parameter at the top resonance.
\end{proof}

\begin{rem} \label{haasemin}
The problem of finding rational points on an ellipsoid is itself interesting from the point of view of number theory. In this regard, an important result concerning rational quadratic forms is the Haase-Minkowski theorem (we refer the reader to section 1.7 of \cite{borevich} for an introduction to this topic) which states that a rational quadratic form represents zero in the field of rational numbers if and only if it represents zero in the field of real numbers and in all fields of $p$-adic numbers. 

By considering the indefinite quadratic form $d_1 \alpha_1^2+ \ldots + d_l \alpha_l^2-z^2$, the Haase-Minkowski theorem can be used to show that if $l\geq 4$ there always exist rational points on the ellipsoid. Also, it is a standard result in number theory that  for $l \geq 2$, whenever it is non-empty, the set of rational points on the ellipsoid is infinite and dense (in the metric topology). 

When $l=2,3$ the Haase-Minkowski theorem can be used to rule out cases where the ellipsoid has no rational points. For example, if $l=3$ and $d_i=7$ for $i=1,2,\, 3$, we obtain the quadratic form $7 \alpha_1^2+7\alpha_2^2+7\alpha_3^2-z^2$. This form does not represent zero in the rational numbers. To see this, suppose that there exists an integer solution to $7 \alpha_1^2+7\alpha_2^2+7\alpha_3^2=z^2$. Dividing both sides by 4 as necessary, we can assume that not all the $\alpha_i$'s are even. Since every square is congruent to 0, 1, or 4 modulo 8, trying out all the combinations shows that the equation above has no solutions modulo 8.
\end{rem}

We summarize the results we obtained in this section as follows. For the $\alpha_1>-2$ case we found that formal series solutions were parametrized by the set of rational points on the ellipsoid $d_1 \alpha_1^2 + \ldots +d_l \alpha_l^2=4$. Solutions form an $r+3$-parameter family and therefore these series do not pass the weak Painlev\'e test. In all cases, compatibility conditions hold at the top resonance $\iota=2$. In the context of Ricci solitons, after accounting for the parameters $t_0$, $b_0$, and the parameter at the top resonance $\iota=2$, we obtain an $r$-parameter family of formal solutions that correspond to steady Ricci solitons.  Finally, from the proof of Proposition \ref{ellipsol} we see that the case $r=l=2$, $\alpha_1=\alpha_2=-1$, $d_1=d_2=2$ is singled out as the only one where expansions on integer powers of $t$ exist.

\section{The B\'erard Bergery ansatz}\label{berard}
We now turn our attention to the analysis of the equations corresponding to the B\'erard Bergery ansatz. Recall that in this case the underlying manifold is the total space of an $S^1$ bundle over a Fano K\"ahler-Einstein base and the steady Ricci soliton equation is equivalent to \eqref{xv}. We look for solutions of the form: 
\begin{equation} \label{ansatzbergery}
\begin{aligned}
x_j&=\sum_{i=0}^\infty a_{j,i}\, t^{\alpha_j+iQ}   && \qquad v_j= \sum_{i=0}^\infty c_{j,i}\, t^{\gamma_j+iQ} \qquad \text{for } 1\leq j\leq 2  \\ 
x_{3}&= \sum_{k=0}^\infty b_k\, t^{\beta+kQ}  && \qquad v_{3}=\sum_{k=0}^\infty e_k\,  t^{\epsilon+kQ},
\end{aligned}
\end{equation}
where $Q$ is a rational number to be determined. The possible leading order terms are the following.
\begin{subequations}
\begin{align}
a_{1,0} \alpha_1 t^{\alpha_1 -1} \, &: \, -2 a_{1,0} c_{1,0}\, t^{\alpha_1 + \gamma_1}  \label{a1}\\
a_{2,0} \alpha_2 t^{\alpha_2 -1} \, &: \, -4 a_{2,0} c_{1,0} \,t^{\alpha_2 + \gamma_1}, \, -2 a_{2,0} c_{2,0} t^{\alpha_2 + \gamma_2} \label{a2}\\
b_0 \beta t^{\beta -1} \, &: \, b_0 e_0 t^{\beta + \epsilon} \label{b1} \\
c_{1,0} \gamma_1 t^{\gamma_1 -1 } \, &: \, -c_{1,0} e_0 t^{\gamma_1 + \epsilon}, \,\, d_2^{-1}a_{1,0}\, t^{\alpha_1}, \,\, 2 d_2^{-1}a_{2,0} \, t^{\alpha_2}  \label{c1}\\
c_{2,0} \gamma_2 t^{\gamma_2 -1 } \, &: \, -c_{2,0} e_0 t^{\gamma_2 + \epsilon}, \,\, a_{2,0} t^{\alpha_2}  \label{c2}\\
e_0 \epsilon t^{\epsilon -1 } \, &: \, -d_2 c_{1,0}^2t^{2\gamma_1} , \, -c_{2,0}^2t^{2\gamma_2} \label{d1}
\end{align}
\end{subequations}
The behaviour of the leading order terms is similar to the $r>1$ case for warped products, although \eqref{a2} and \eqref{c1} introduce some additional complications. We have the following lemma.
\begin{lem}
The leading order terms satisfy points (1)-(5) in Lemma \ref{leadorder}. Additionally, the following properties hold:
\begin{enumerate} \setcounter{enumi}{5}
	\item There exists an $i$ for which $\gamma_i=-1$.
	\item If $\alpha_i=-2$ then $\gamma_i=-1$.
	\item If $\alpha_1 = 0$ then $\gamma_1=1$, $\gamma_2=-1$, and $\alpha_2=2$.
	\item If $\alpha_1=-2$ then $\alpha_2=-2$.
	\item $\alpha_1=0$ if and only if $\gamma_1>-1$.
	\item If $\gamma_1 < \gamma_2$ then $2\alpha_1=\alpha_2$.
\end{enumerate} 
\end{lem}
\begin{proof}
The proof for (1)-(5) is the same as before. (6) follows easily from \eqref{a1} and \eqref{a2}. For (7) note that $\alpha_1 = -2$ implies $\gamma_1 =-1$ by \eqref{a1}. If $\alpha_2 =-2$, then \eqref{c2} gives $\gamma_2=-1$. To prove (8) note that if $\alpha_1=0$ \eqref{a1} implies that $\gamma_1>-1$ and consequently $\gamma_2=-1$. \eqref{c1} implies that $\alpha_2>-2$ and so $e_0=1$. By \eqref{d1} $c_{2,0}=-1$ and by \eqref{a2} $\alpha_2=2$. Using \eqref{c1} we deduce that $\gamma_1=1$. To prove (9) suppose by contradiction that $\alpha_2>-2$. If $\gamma_2=-1$ then \eqref{c1} and \eqref{c2} give incompatible values for $e_0$. On the other hand, if $\gamma_2>-1$ then \eqref{a1} gives $c_{1,0}=1$ and then \eqref{a2} shows $\alpha_2=-4$, a contradiction. Finally, (10) and (11) are obvious from \eqref{a1} and \eqref{a2}.
\end{proof}
The lemma above greatly reduces the admissible combinations of leading exponents. However, there are still many more combinations than in the case for warped products. To simplify the computations we will focus exclusively on one of the cases identified by the previous lemma, that is, when $\alpha_1 = 0$, $\alpha_2=2$, $\gamma_1=1$, and $\gamma_2=-1$. The reason for focusing solely on this type of expansion will become apparent in the next section. It is easy to see that whole list of leading terms is given by
\begin{equation} \label{leadorfamV}
\begin{array}{llll}
a_{1,0}=\textnormal{Arbitrary} & c_{1,0}= a_{1,0}/2d_2 \quad & \alpha_1=0 \quad & \gamma_1=1 \\
a_{2,0}=\textnormal{Arbitrary} & c_{2,0}=-1 &\alpha_2=2 \quad & \gamma_2=-1 \\
b_0=\textnormal{Arbitrary} & e_0=1\quad & \beta=1 \quad & \epsilon=-1.
\end{array}
\end{equation}
Since all the leading exponents are integers we can set $Q=1$ in \eqref{ansatzbergery}. The recursion relation is now given by
\begin{equation}\label{famV}
\begin{aligned}
 \left(  \begin{array}{ccc|ccc}
i  & & & 0 & & \\
 & i & & & 2a_{2,0} & \\
 & & i & & & -b_0 \\ \hline
-1/d_2 &  &  & i+2 & & c_{1,0} \\
 & 0 & & & i & -1 \\
 & & 0 &    & -2 & i-1 \\
\end{array} \right)  &
\left(  \begin{array}{c}
a_{1,i} \\
a_{2,i}\\
b_i\\ \hline
c_{1,i} \\
c_{2,i}\\
e_i
\end{array} \right)=\\
\left( \begin{array}{c}
-2\sum_{k=0}^{i} a_{1,i-2-k}\, c_{1,k} \\
-4\sum_{k=0}^{i} a_{2,i-2-k}\,c_{1,k}-2\sum_{k=1}^{i} a_{2,i-k}\,c_{2,k}\\
\sum_{k=1}^{i-1} b_{i-k} \, e_k\\ \hline
-\sum_{k=1}^{i-1} c_{1,i-k}\, e_k+2d_2^{-1}a_{2,i-2} \\
-\sum_{k=1}^{i-1} c_{2,i-k}\, e_k+a_{2,i-4} \\
-d_2 \sum_{k=0}^{i} c_{1,i-4-k}\, c_{1,k} - \sum_{k=1}^{i-1} c_{2,i-k} \, c_{2,k}  \end{array} \right)&.
\end{aligned}
\end{equation}
The determinant $\det X(i)$ factorises as
$$\det X(i)=i^3(i+2)(i+1)(i-2),$$
so that we have resonances at $i=-1,0,$ and $2$, with the root on $i=-2$ carrying no meaning as usual. As before, it is easy to check that compatibility conditions at the top resonance are satisfied and we obtain a 5-parameter family of solutions of \eqref{xv}. This means that the system does not pass the full Painlev\'e test in this cas. In terms of the geometric interpretation, we can choose the top parameter so that we obtain solutions to the Ricci soliton equation and we therefore obtain a 2-parameter family of Ricci solitons.
\begin{pro}
The series expansions corresponding to the  B\'erard Bergery ansatz satisfy compatibility conditions at the top resonance $i=2$. Moreover, the free parameter at the resonance can be chosen so that the series lie on the hypersurface $H=0$.
\end{pro}
\begin{proof}
The right hand side of the recursion \eqref{famV} vanishes when $i=1$. For $i=2$, the vector on the right is 
$$ (a_{1,0}c_{1,0}, -4 a_{2,0}c_{1,0}, 0, 2a_{2,0}/d_2,0,0).$$
The kernel $\ker X(2)^T$ is spanned by $(0,0,0,0,1,1)$ which readily implies the vector above is in the image of $X(2)$. An argument similar to the ones above shows that the free parameter can be chosen so that the solution lies on $H=0$.
\end{proof}

\section{Convergence of formal series and conclusions}\label{conclusions}

The last step remaining in the Painlev\'e analysis of the solitons described above is to verify that the formal series converge in a punctured neighbourhood of the origin. This is readily achieved since the nonlinearities in both \eqref{xu} and \eqref{xv} are quadratic and therefore the majorization argument in section 6 of \cite{DW2001} can be employed to obtain upper bounds for the absolute value of coefficients at each step of the recursion. Therefore, the series obtained in the previous sections converge for small enough values of the argument $t$. We conlcude our study of the Ricci soliton equatations with some geometric considerations.

\subsection{Warped products}
It is illuminating to study the solitons singled out by the Painlev\'e analysis in the context of the variables used in \cite{DW2009}. Here, the metric was written (compare with \eqref{metric}) as
$$ \bar{g}= dt^2+ \sum_{i=1}^{r} h_i(t)^2 \, g_i$$
and coordinates were introduced by taking 
\begin{equation} \label{oldvariables}
 X_i:= \frac{\sqrt{d_i}}{-\dot{u}+\tr L } \, \frac{\dot{h}_i}{h_i}, \quad  \qquad  Y_i:= \frac{1}{-\dot{u}+\tr L}\, \frac{1}{h_i},
\end{equation}
where $L=L_t$ is the shape operator of the hypersurfaces $P_t $ (notice we are normalizing the scalar curvature of each factor to be equal to one, so in this case the Einstein constant of each factor is $\lambda_i=1/d_i$), and $u$ is the soliton potential. We remark that in \cite{DW2009} $X_i$ and $Y_i$ are functions of a new independent parameter $s$ which is introduced by letting $ds=(-\dot{u}+\tr L) \, dt$. This new parameter satisfies that $s\rightarrow -\infty$ as $t \rightarrow 0$. In these variables the Ricci soliton equation is shown  to be equivalent (see (2.6) and (2.7) in \cite{DW2009}) to the flow of
\begin{equation} \label{vectfield}
\begin{aligned}
X_i' &= X_i \left( \sum_{j=1}^r X_j^2 -1 \right) + \frac{Y_i^2}{\sqrt{d_i}}\\
Y_i' &= Y_i \left( \sum_{j=1}^r X_j^2 - \frac{X_i}{\sqrt{d_i}} \right),
\end{aligned}
\end{equation}
where the prime denotes differentiation with respect to $s$. In the variables used for the Painlev\'e analysis we have that $x_i=e^{-q_i}=h_i^{-2}$ for $1\leq i \leq r$ and $x_{r+1}=e^{\frac{1}{2} \dd \cdot \qq}$, so the variables \eqref{oldvariables} can be written, considered as functions of the geodesic length $t$, as
$$ X_i= \frac{\sqrt{d_i}}{2}\, \frac{\dot{x}_i}{x_i} \frac{x_{r+1}}{\dot{x}_{r+1}} \quad \qquad Y_i= \sqrt{x_i}\, \frac{x_{r+1}}{\dot{x}_{r+1}}, $$ 
where differentiation with respect to $t$ is denoted by a dot. Thus, expressions for $X_i$ and $Y_i$ are given by
\begin{align*}
X_i &= -\frac{\sqrt{d_i}}{2}\, \frac{\alpha_i \,a_{i,0} \, b_0 \, t^{\alpha_i+\beta-1} + \ldots }{\beta \, a_{i,0} \, b_0 \, t^{\alpha_i+\beta-1} + \ldots}\\[1mm]
Y_i &= \frac{\sqrt{a_{i,0}} \, b_0 \, t^{\frac{\alpha_i}{2}+\beta}+ \ldots}{\beta \, b_0 \, t^{\beta-1}+ \ldots}.
\end{align*}
Using these expressions it is possible to show that our series solutions emanate from various different equilibrium points of the vector field \eqref{vectfield}. In particular, by analyzing the asymptotic behaviour of the trajectories as $t \rightarrow 0$ we can see that the leading exponents $\alpha_i$ determine an equilibrium of \eqref{vectfield}. As before, we notice different behaviours depending on whether $\alpha_1=-2$ or $\alpha_1>-2$. 

The expansions corresponding to the $\alpha_1=-2$ case satisfy that 
$$ X_i\rightarrow \frac{\sqrt{d_i}}{e_0} \qquad \textnormal{and }  \quad \quad Y_i \rightarrow \frac{\sqrt{d_i(e_0-1)}}{e_0} \qquad \textnormal{if } 1\leq i \leq l,$$ 
and $X_i \rightarrow 0$, $Y_i \rightarrow 0$ if $l+1 \leq i \leq r$. Recalling that $e_0=\sum_{k=1}^l d_k$ for this type of series, it is easy to check that these coordinates correspond to an equilibrium of \eqref{vectfield}. 
\begin{rem}
In the particular case where $\alpha_1=-2$, $l=1$, and $r\geq 1$ the series singled out by the Painlev\'e test correspond to trajectories emanating from $X_1=1/\sqrt{d_1}$, $Y_1=\sqrt{1-\frac{1}{d_1}}$, These are precisely the trajectories analyzed in \cite{DW2009} and yield solitons which are complete and smooth at the origin (see Theorem 4.17).
\end{rem}
On the other hand, the asymptotic behaviour as $t \rightarrow 0$ of trajectories corresponding to the case $\alpha_1>-2$ is given by
$$ X_i\rightarrow -\frac{\alpha_i \, \sqrt{ d_i}}{2} \qquad \textnormal{and }  \quad \quad Y_i \rightarrow 0 \qquad \textnormal{if } 1\leq i \leq l,$$ 
and $X_i \rightarrow 0$, $Y_i \rightarrow 0$ if $l+1 \leq i \leq r$. The fact that leading exponents $\alpha_i$ lie on the ellipsoid immediately implies that these trajectories also emanate from an equilibrium of \eqref{vectfield}. 
\begin{rem}
When $\alpha_1>-2$, $l=1$, and $r\geq 1$ the series coming from the Painlev\'e analysis correspond to trajectories emanating from $X_1=\pm 1$, $Y_1=0$, and $X_i=Y_i=0$ for $i>1$. These are the trajectories analysed in \cite{ACF2015} (cf. A.1 and A.2). They correspond to solitons that are complete but not smooth at the origin (if $X_1=1$) or that blow up at the origin (if $X_1=-1$).
\end{rem}

Note that in both cases the equilibrium points lie on the zero set $\mathcal{L}=0$ of the Lyapunov function  $\mathcal{L}:=\sum_{i=1}^r X_i^2+Y_i^2-1$ introduced in \cite{DW2009}. The asymptotic behaviour at the other end is determined by the following proposition.

\begin{pro}
The series obtained in the Painlev\'e analysis correspond to trajectories of \eqref{vectfield} contained in the region $\mathcal{L}<0$ of the Lyapunov function.
\end{pro} 
\begin{proof}
Using \eqref{xu} we write $-2u_i=\dot{x}_i/x_i$ for $1 \leq i \leq r$ and $u_{r+1}=x_{r+1}/\dot{x}_{r+1}$. Substituting in the above expressions for $X_i$ and $Y_i$, and using the fact that series solutions lie on the level set $H=0$ we get that
$$ \sum_{i=1}^r X_i^2+Y_i^2= \frac{1}{u_{r+1}^2}\left( \sum_{k=1}^r d_k \, u_k^2 + x_k \right)= 1-\frac{1}{u_{r+1}^2},$$
from which the result follows.
\end{proof}
Using this we conclude that solutions coming from the Painlev\'e analysis are complete at infinity, since Proposition 3.7 in \cite{DW2009} implies that the corresponding trajectories converge to the origin in the $X_i$, $Y_i$ coordinates.\\

We finalize by relating the results of the Painlev\'e test to what is known about Ricci solitons on multiple warped products. First, as we pointed out in section \ref{onefactor}, when $r=1$ solutions corresponding to the $\alpha=-2$ case do not pass the weak (and therefore strong) Painlev\'e test. However, the resulting series are meromorphic functions of $t$ so we might find an integrable subsystem of \eqref{xu} corresponding to these expansions. On the other hand, expansions with $\alpha>-2$ have 4 free parameters if $d_1$ is a perfect square, which means that these cases pass the weak Painlev\'e test. Additionally, when $d_1=4$ the resulting series are meromorphic functions of $t$, so in this case we pass the strong Painlev\'e test. 

When $r>1$ system \eqref{xu} never passes the weak (and therefore the strong) Painlev\'e test. However, the series corresponding to the $\alpha_1=-2$ are meromorphic functions of $t$, which suggest that a subsystem might be integrable. Likewise, when $\alpha_1>-2$ the weak Painlev\'e test fails but it does single out the case $r=l=2$, $d_1=d_2=1$ as the only one for which there are solutions meromorphic on $t$.

In the context of Ricci solitons, in \cite{BDW2016} explicit closed form solutions were found for solitons in two cases. The first one occurred when there was exactly one factor of dimension 4 in the product, and the integration of the equations was accomplished by obtaining a conserved quantity for the Hamiltonian. The second was when there were two factors in the product, each of dimension two. In these case, the equations were integrated by using a superpotential for the Hamiltonian. It is interesting to note that these are exactly the cases singled out by the Painlev\'e analysis. In particular, the 4-dimensional case with one factor is the only one that passes the strong Painlev\'e test (when $\alpha>-2$). On the other hand, the case with $r>1$ never passes the Painlev\'e test (weak or strong) however, the case which comes closest to passing the strong test (in the sense of having the greatest number of free parameters) is exactly the one where there are two factors, each of dimension 2. The results of the test therefore suggest that no closed form solutions for Ricci solitons exist in warped products aside from the ones found in \cite{BDW2016}.

\subsection{B\'erard Bergery ansatz}

As described in section \ref{setup}, the B\'erard Bergery ansatz gives us metrics where the hypersurface is the total space of an $S^1$ bundle over a Fano K\"ahler-Einstein base $(B, g_B)$. In this case, the metric can be written as
$$ \bar{g}=dt^2+ f(t)^2 \theta \otimes \theta + g(t)^2 \pi^*g_B, $$
and the Ricci soliton equation becomes a second order system of ODE's for $f$, $g$ and the soliton potential $u$. This is the system studied in \cite{DW2011} (cf. section 4). Comparing with the variables introduced for the Painlev\'e analysis in section \ref{setup}, we see that these are related by
$$ f^2=k_1 \, x_2 x_1^{-2} \quad \textnormal{and} \quad  g^2=k_2 \, x_1^{-1}$$
for some constants $k_1$ and $k_2$ (again, the exact values are irrelevant). In order to obtain a complete metric, we need to compactify the metric a the $t=0$ end. This can be done in two ways. The first one is by collapsing the $S^1$ in each fibre and adding a copy of $B$ at $t=0$. In terms of $f$ and $g$ this means that we need solutions that satisfy $f(0)=0$, $\dot{f}(0)=1$, $g(0)>0$, and $\dot{g}(0)=0$. The second way in which we can compactify at the $t=0$ end is if $(B,g_B)$ is the complex projective space $\mathbb{C}P^{d_2}$ with the Fubini-Study metric. In this case we can use the Hopf fibration $S^{2d_2+1}\rightarrow \mathbb{C}P^{d_2}$, where the sphere has constant curvature equal to one, and let both factors collapse. In this case, we just need to add a point at the origin and the resulting manifold is $\mathbb{C}^{d_2+1}$. In terms of the functions, this requires us to look for solutions that satisfy $f(0)=0$, $\dot{f}(0)=1$, $g(0)=0$, and $\dot{g}(0)=\pm 1/\sqrt{2}$.

If we substitute the series for $x_1,\, x_2$ arising from the recursion (it is enough to calculate up to step $i=2$), it is easy to see to see that in both cases previous necessary conditions for smoothness at the origin are satisfied if we take $\alpha_1=0$, $\alpha_2=2$, and the free parameters $a_{1,0},\, a_{2,0}$ are related by
$$ \sqrt{\frac{a_{2,0}}{a_{1,0}^2}}=\frac{1}{k_1}.$$
Thus, any solution that can be completed by compactifying at the origin has to be contained in the family of series that satisfy $\alpha_1=0$, $\alpha_2=-2$.  
As mentioned in the introduction, steady K\"ahler Ricci solitons have been found in the context of $S^1$ bundles over a Fano K\"ahler-Einstein base in \cite{Cao1996}, \cite{FIK2003}, and \cite{DW2011}. In particular, the K\"ahler Ricci solitons in \cite{Cao1996} can be explicitly integrated (see Theorem 4.20 in \cite{DW2011} and the discussion prior to it). Interestingly, \cite{Wink2017} and \cite{Sto2017} have recently obtained existence results for non-K\"ahler solitons on these spaces (see also \cite{Appleton2017}). The failure of the Painlev\'e test suggests that the Ricci soliton equation on these complex line bundles falls short of full integrability. It would be interesting to determine whether the series expansions singled out by the Painlev\'e test yield an integrable subsystem where solutions correspond to these non-K\"ahlers steady solitons on complex line bundles.


\begin{thebibliography}{9}

\bibitem{ARS1980} Ablowitz, M.J., Ramani, A., Segur, H., A connection between nonlinear evolution equations and ordinary differential equations of P-type. I, J. Math. Phys. 21 (1980), no. 4, 715--721. 

\bibitem{ARS1980b} Ablowitz, M.J., Ramani, A., Segur, H., A connection between nonlinear evolution equations and ordinary differential equations of P-type. II, J. Math. Phys. 21 (1980), no. 5, 1006--1015.

\bibitem{ACF2015} Alexakis, S., Chen, D., Fournodavlos, G., Singular Ricci solitons and their stability under the Ricci flow,  Comm. Partial Differential Equations 40 (2015), no. 12, 2123–2172. 

\bibitem{Appleton2017} Appleton, A., A family of non-collapsed steady Ricci solitons in even dimensions greater or equal to four, arXiv:1708.00161

\bibitem{BB1982} B\'erard-Bergery, L., Sur des nouvelles vari\'et\'es riemanniannes d'Einstein, Inst. \'Elie Cartan, 6, Univ. Nancy, 1982. 

\bibitem{besse} Besse, Arthur L. \textit{Einstein manifolds},
Ergebnisse der Mathematik und ihrer Grenzgebiete (3), 10. Springer-Verlag, Berlin, 1987. xii+510 pp

\bibitem{BDW2016} Betancourt, A., Dancer, A., Wang, M., A Hamiltonian approach to the cohomogeneity one Ricci soliton equations and explicit examples of non-K\"ahler solitons, J. Math. Phys. 57 (2016), no. 12, 122501, 17 pp. 

\bibitem{borevich} Borevich, Z., Shafarevich, \textit{Number Theory}, Academic Press, 1st ed., (1966)

\bibitem{bryant} Bryant, R., Ricci flow solitons in dimension three with $SO(3)$ symmetry, unpublished.

\bibitem{Cao1996} Cao, Huai-Dong, Existence of gradient Kähler-Ricci solitons, Elliptic and parabolic methods in geometry, 1--16, 1996. 

\bibitem{DW2001} Dancer, A., Wang, M., The cohomogeneity one Einstein equations and Painlev\'e analysis,  J. Geom. Phys. 38 (2001), no. 3-4, 183--206. 

\bibitem{DW2001b} Dancer, A., Wang, M., Ricci-Flat warped products and Painlev\'e analysis, J. Math. Phys.,  J. Math. Phys. 42 (2001), no. 8, 3609--3614. 

\bibitem{DW2003} Dancer, A., Wang, M., Painlev\'e expansions and the Einstein equations: the two-summand case,  J. Geom. Phys. 48 (2003), no. 1, 12--43. 

\bibitem{DW2003b} Dancer, A., Wang, M., Painlev\'e analysis of the Ricci-flat ordinary differential equations associated with Aloff–Wallach spaces
and $U(1)$-bundles over Fano products. Integrability, topological solitons and beyond. J. Math. Phys. 44 (2003), no. 8, 3383--3406.
    
\bibitem{DW2009} Dancer, A., Wang, M., Some new examples of non-K\"ahler Ricci solitons, Math. Res. Lett.,  Math. Res. Lett. 16 (2009), no. 2, 349--363.

\bibitem{DW2011} Dancer, A., Wang, M.On Ricci solitons of cohomogeneity one,
Ann. Global Anal. Geom. 39 (2011), no. 3, 259--292.  
 
\bibitem{EW2000} Eschenburg, J.-H.; Wang, McKenzie Y.
The initial value problem for cohomogeneity one Einstein metrics., J. Geom. Anal. 10 (2000), no. 1, 109--137.

\bibitem{FIK2003} Feldman, M.; Ilmanen, T.; Knopf, D.,
Rotationally symmetric shrinking and expanding gradient Kähler-Ricci solitons, J. Differential Geom. 65 (2003), no. 2, 169--209.  
 
\bibitem{Ivey1994} Ivey, T., New examples of complete Ricci solitons,  Proc. Amer. Math. Soc. 122 (1994), no. 1, 241--245. 

\bibitem{Koi1990} Koiso, N., On rotationally symmetric Hamilton's equation for Kähler-Einstein metrics, Recent topics in differential and analytic geometry, 327--337, Adv. Stud. Pure Math., 18-I, Academic Press, Boston, MA, 1990. 

\bibitem{Pa1978} Page, D., A compact rotating gravitational instanton, Phys. Lett. B,   79, Issue 3, (1978),  235--238

\bibitem{PP1987} Page, D.; Pope, C. N., Inhomogeneous Einstein metrics on complex line bundles, Classical Quantum Gravity 4 (1987), no. 2, 213--225. 

\bibitem{perelman1} Perelman, G., The entropy formula for the Ricci flow and its geometric applications, arXiv:math/0211159

\bibitem{Sto2017} Stolarski, M., Steady Ricci Solitons on Complex Line Bundles, arXiv:1511.04087

\bibitem{WZ1986} Wang, M., Ziller, W., Existence and nonexistence of homogeneous Einstein metrics, Invent. Math. 84 (1986), no. 1, 177–194. 

\bibitem{Wink2017} Wink, M., Cohomogeneity one Ricci Solitons from Hopf Fibrations, arXiv:1706.09712
      
\end{thebibliography}
\end{document}